\documentclass[11pt]{article}
\usepackage{amsmath}
\usepackage{amsfonts}
\usepackage{graphicx}
\usepackage{color}
\usepackage{geometry}
\usepackage{graphics}
\usepackage{amsmath}
\usepackage{amsfonts}
\usepackage{amssymb}
\usepackage{float}
\usepackage{subfigure} 
\usepackage{algorithm}
\usepackage{algorithmicx}
\usepackage{amsthm}

\newtheorem{theorem*}{Theorem}
\newtheorem{example}{Example}
\newtheorem{corollary}{Corollary}
\newtheorem{definition*}{Definition}

\newtheorem{remark}{Remark}
\newtheorem{proposition*}{Proposition}

\newcommand{\m}[1]{\ensuremath{\mathbf{#1}}}

\newcommand{\x}{\m{x}}
\newcommand{\pp}{p_1}
\newcommand{\dd}{\vec{d}}
\newcommand{\FR}{\tilde{R}}

0in \numberwithin{equation}{section}

\title{A direct probing method of an inverse problem\\ for the Eikonal equation}

\author{
Kazufumi Ito\footnote{Department of Mathematics and Center for Research in Scientific Computation, North Carolina State University, Raleigh, NC 27695. (kito@ncsu.edu).}
\and Ying Liang \footnote{Department of Mathematics, Purdue University, West Lafayette, IN 47907. (liang402@purdue.edu).}}

\begin{document}

\maketitle

\begin{abstract}
In this paper, we propose a direct probing method for the inverse problem involving the Eikonal equation.
For the point-source Eikonal equation,  the viscosity solution represents the least travel time of wave fields from the source to the point at the high-frequency limit.
 The corresponding inverse problem is to determine the inhomogeneous wave-speed distribution from the first-arrival time data at the measurement surfaces corresponding to distributed point sources.
We analyze the Eikonal inverse problem and show that it is highly ill-posed.
Then we develop a direct probing method that incorporates the solution analysis of the Eikonal equation and several aspects of the velocity models.
When the wave-speed distribution has a small variation from the homogeneous medium, we reconstruct the inhomogeneous wave-speed distribution using the filtered back projection method. For the high-contrast media, we assume a  background medium and develop an adjoint-based back projection method to identify the variations of the medium from the assumed background.

\end{abstract}

\section{Introduction}
The Eikonal equation plays an important role in a wide range
of applications such as geometrical optics, seismic imaging, and computer vision.
It results from the substitution of the ray series solution into the elastic equation of motion \cite{erven1977RayMI, cerveny1985application, cerveny1987ray}, and the leading term in the high-frequency limit gives the P- and S-particle motions and travel times.
The point-source Eikonal equation computes the first arrival (traveling) time from the source to a location.
The corresponding inverse problem,
 Eikonal tomography, is defined as using 
first arrival time data corresponding to a set of distributed point sources observed at the surface to reconstruct the wave speed of the medium \cite{sei1994gradient, sei1995convergent, leung2006adjoint}.  In the high-frequency regime, the slowness is inversely proportional to the wave speed of the medium.
In the low-frequency regime, one can introduce a frequency-dependent correction term to the Eikonal equation, and obtain a WKB reconstruction of the inhomogeneous velocity medium.  
That is, the velocity field reconstruction modeled by the Eikonal equation actually estimates an effective slowness function (not simply defined by the wave speed),  which is accurate only in the high-frequency regime.
The inverse problem for the Eikonal equation is based on the mathematical ideology that the first travel time can be computed by the Eikonal equation given the wave speed. The first-arrival travel-time data also corresponds to the phase measurement for the Helmholtz equation in the high-frequency regime.
Once the inhomogeneous velocity field is reconstructed, one can recover the amplitude data of the solution to the Helmholtz equation with the reconstructed inhomogeneous medium. 

The Eikonal tomography problem is highly ill-posed in general.
Our objective is to develop a two-step direct probing method that incorporates several aspects of Eikonal models
and the solution analysis of the Eikonal equation.
We propose an effective and efficient reconstruction algorithm for the inhomogeneous medium
 utilizing the traveling time data or phase data of the transmission wave.
 The algorithm is motivated by the relationship between fanbeam tomography and Eikonal tomography, i.e., 
the observation that the linearization of the Eikonal equation for the homogeneous background media 
reduces to the straight lines ray tracing, which can be modeled by the line integrals of fanbeam transformation. 

  To efficiently calculate the inverse fanbeam transform in the two-step direct probing method that we propose, we present a filtered back projection algorithm.  In various diagnostic imaging modalities, parallel projection and fan-beam projection are two common acquisition geometries, and the parallel projection is computed as the Radon transform. We present an algorithm for the inverse fanbeam transform in Section \ref{sec:4} by extending the filtered back projection algorithm for the inverse Radon transform. 
The method is based on the inverse radon transform theory \cite{kak2001principles, natterer2001mathematics,helgason1980radon}. It consists of the adjoint fanbeam transform and the regularized ramp filter of the difference between measurement data and a background solution of the Eikonal equation.
It is shown and analyzed that for the low-contrast media, our proposed two-step direct probing method works very well and obtains a very sharp reconstruction.

For the high-contrast media, we assume a background medium and apply the adjoint-based back projection method to identify variations of the medium from the assumed background. 
In order to compute the back projection, we use the adjoint equation corresponding to the assumed background medium.
That is, we compute the direction field of the Eikonal equation at the assumed background to derive the adjoint transform of the linearized forward transform. We mollify the direction field via regularization for the high contrast or discontinuous background medium to achieve a more robust back projection, which  is a more stable numerical method than the ray-tracing method \cite{berryman1990stable, bishop1985tomographic, washbourne2002crosswell}.

To solve the ill-posed traveltime tomography problem, traditional methods are mostly based on ray-tracing, which involves tracing the ray path by computing the solutions to the ray equations \cite{nolet2008breviary}.  There have also been alternate approaches that avoid explicit ray tracing. Sei and Symes \cite{sei1994gradient, sei1995convergent} utilize an adjoint-state method for traveltime tomography based on paraxial eikonal equations, and in \cite{leung2006adjoint}, the Eikonal tomography is formulated as a constrained optimization problem, then the adjoint equation-based gradient method is developed for the reconstruction of the inhomogeneous medium.  Various improved variational method has been proposed \cite{li2013fast, li2014level,taillandier2009first}.
Our contribution and innovation are that we develop a direct sampling method
to remedy the ill-posedness of the Eikonal inverse problem and to develop efficient 
and effective probing algorithms.  We show that the linearization of the Eikonal equation at the homogeneous medium reduces to the fanbeam (cone) transform. We then develop the filtered back projection method of the  fanbeam inverse tomography problems to probe the inhomogeneous velocity field for a low contrast media. Next, we develop a direct probing method for high contrast cases with assumed background and the compute the back projection by the adjoint equation for the linearized Eikonal operator. 
Our method can also be applied to different geometrical and bio-medical tomography.
We improve and extend the application of computational tomography to wave media inverse problems.  As done for the inverse medium problem \cite{ito2013two, ito2022least},
our reconstruction can serve as an initialization of optimization-based reconstruction methods
to improve the reconstruction and remedy the numerical ill-posedness and complexity.

The rest of the paper is structured as follows. In Section \ref{sec:2} we introduce the Eikonal equation as a model of scattering problems and its relation to the fanbeam tomography.  
We analyze the approximation error of using fanbeam model to approximate the Eikonal model.
 The analysis leads to the algorithm to transform Eikonal 
measurements to fanbeam sinogram, and then solve the better-posed inverse fanbeam problem.
 In Section \ref{sec:fbp}, we introduce the pre-filtered back projection method for fanbeam tomography and apply it to Eikonal tomography.
In Section \ref{sec:4}, we introduce the adjoint operator of the linearized Eikonal equation and analyze the adjoint equation.
In Section \ref{sec:5} we present our numerical tests and demonstrate the effectiveness
of the proposed algorithm.

\vspace{2mm}

\section{Eikonal Equation and inverse Eikonal problem}\label{sec:2}
In this section, we review the derivation of the Eikonal equation, illustrate the ill-posedness of the inverse Eikonal problem, and introduce the fan beam tomography as a linearization of the inverse Eikonal problem.
 Consider a bounded Lipschitz domain $\Omega\subset \mathbb{R}^2$
 and the Eikonal equation 
\begin{equation}\label{eq:Eikonal}
|\nabla u(\m{x})|=f(\m{x}),\;\; u(\m{x}_0)=0,
\end{equation}
where $f\ge1$ denotes the slowness function. The unique viscosity solution \cite{kruvzkov1967generalized,lions1982generalized, bardi1997optimal}   of \eqref{eq:Eikonal} given $f\in W^{1,\infty}(\Omega)$ denotes the least travel-time from the source $\m{x}_0$ to an arbitrary point connected by a shortest ray-path.
In general, solutions to an Eikonal equation are not unique
and define the secondary reflection from the medium. In this work, we only consider the first-arrival based traveltime tomography, that is, given both the first-arrival travel-time measurements $u(\m{x})$ on the boundary $\Gamma=\partial\Omega$ and the location of the point source $\m{x}_0$, the goal is to reconstruct the slowness function $f$ inside the domain.

\subsection{High-frequency limit of time harmonic wave equation}
The Eikonal equation can be derived from the wave equation. Consider the following wave equation 
$$u_{tt}-v^2(\m{x})\Delta u=0.$$
Taking the Fourier transform in $t$, one obtains the time-harmonic equation:
\begin{equation} \label{Helm}
\Delta \psi+\frac{\omega^2}{v^2(\m{x})}\psi=0,
\end{equation}
where $\omega$ denotes the frequency. Now assuming  that $\psi(\m{x},\omega)$ is a solution to \eqref{Helm} of the form
$$
\psi(\m{x},\omega)=A(\m{x},\omega)e^{i\omega\phi(\m{x},\omega)},
$$
one can calculate the component of the Laplacian operator for each spatial axis $j$:
\begin{equation*}
\partial^2_j\psi=(\partial^2_jA+2i\omega\partial_j A\partial_j\phi+iA\omega\partial^2_j\phi-A\omega^2
(\partial_j \phi)^2)e^{i\omega \phi}.
\end{equation*}
Substituting this into the Helmholtz equation \eqref{Helm}, we have
\begin{equation}\label{eq:expansion}
(|\nabla\phi|^2-\frac{1}{v^2})
-\frac{i}{\omega}(\frac{2}{A}\nabla A\cdot\nabla\phi+\Delta \phi)
-\frac{1}{\omega^2 A}\Delta A=0.
\end{equation}
Note that in the high frequency limit when $\omega\to\infty$, the first term dominates and leads to the Eikonal equation,
\begin{equation}\label{eq:eik}
|\nabla \phi|^2=\frac{1}{v(\m{x})^2}: =( f(\m{x}))^2.
\end{equation}
Thus the Eikonal equation is a phase (only) approximation 
of the Helmholtz equation for sufficiently large frequency, and the approximation is fundamentally valid only in this limit. 
This implies that the Eikonal equation (and many other 
ray-tracing techniques) may only be used when variations in velocity are negligible on spatial scales 
that are comparable to the wavelengths of the propagating waves.
Ordering terms in the equation \eqref{eq:expansion} with respect to real and imaginary parts and multiplying the second term on the lefthand sidze by $A\omega/i$, one obtains the transport equation:
\begin{equation*}
2\nabla A \cdot \nabla \phi + A\Delta \phi = 0,
\end{equation*}
or equivalently,
$$
\nabla\cdot (|A|^2 \nabla\phi)=0.
$$
 From the real part of the equation  \eqref{eq:expansion} we obtain the frequency-dependent Eikonal equation:
\begin{equation} \label{fre}
|\nabla\phi(\m{x},\omega)|^2= \frac{1}{v(\m{x})^2}+\frac{1}{\omega^2}\frac{\Delta A(\m{x},\omega)}{A(\m{x},\omega)},
\end{equation}
which is different from \eqref{eq:eik} with a correction term related to the  frequency $\omega$.

If one considers the Helmholtz equation \eqref{fre} with inhomogeneous refractive index, then the information of the inhomogeneity can be recovered from the effective slowness function  $f$ on the righthand side of equation \eqref{fre}. That is,  solving this inverse Eikonal problem is an approximation of the inverse medium problem 
using phase-only data $\phi$ for the high-frequency regime.

\subsection{Ill-posedness of Eikonal inverse problem}\label{sub:illposed}
The Eikonal inverse problem is severely ill-posed as the slowness function cannot be uniquely determined given the measured travel time on the boundary. This fact can be illustrated with one concrete example. Consider a subdomain $\Omega_0 = (-0.5,0.5)\times (-0.5,0.5)\subset \Omega$, source $\m{x}_0 = (0, -1)$, and a slowness function $f$ satisfying
\begin{equation}\label{illposedex1}
f(\m{x})=\begin{cases}
f_0 \quad&  \m{x}\in\Omega_0\\
1 \quad& \mbox{  otherwise},
\end{cases}
\end{equation}
where $f_0$ is a constant number.
Then, we have the value of solution $u$ to the Eikonal equation \eqref{eq:Eikonal} on $\Gamma_1=\{\m{x} =(x_1,x_2): x_2=1\}$ is given by
$$
u(\m{x})=\left\{ \begin{array}{ll}\sqrt{(|x_1|-0.5)^2+0.25}+1+\frac{\sqrt{2}}{2}& \quad x_1\in(-0.5,0.5)
\\ \\
\sqrt{(|x_1|-0.5)^2+2.25}+\frac{\sqrt{2}}{2} & \quad  0.5\ge |x_1 |\ge 2, \end{array} \right.
$$
regardless of the value of $f_0$ as long as $f_0\ge c>0$ is sufficiently large. One will observe the same phenomenon with point sources distributed along the boundary of $\Omega$ and measurement collected on the boudary $\Gamma$.
Thus the value of $f_0(\m{x})$ for $\m{x}\in (-0.5,0.5)\times (-0.5,0.5)$ can not be determined from the measurement $u|_\Gamma$.
We will numerically illustrate the ill-posedness in Example \ref{exp:1}.

\subsection{Eikonal equation and Fanbeam transform}
In this subsection, we analyze the Eikonal equation to illustrate that the inverse Eikonal problem can be approximated by the inverse Fanbeam transform when the slowness function $f$ has a small variation from the homogeneous background. Then we will propose an algorithm for the reconstruction of the slowness function $f$.

 Let $u$ be the viscosity solution to the Eikonal equation \eqref{eq:Eikonal} with $f\in W^{1,\infty}(\Omega)$.
 Let $\overline{u}$ be the viscosity solution to the Eikonal equation \eqref{eq:Eikonal} with the constant background slowness function, that is, $\overline{u}$ solves
 $$|\nabla \overline{u}(\m{x}) |= 1,\quad \overline{u}(\m{x}_0) = 0.$$
Then 
$$
\overline{u}(\m{x})=|\m{x}-\m{x}_0|
$$
defines the shortest time to travel from $\m{x}_0$ to $\m{x}$ in a homogeneous medium. 
We linearize the left hand side of the Eikonal equation  \eqref{eq:Eikonal} at $\overline{u}$, and 
denote its solution by $ u_1$, i.e., 
\begin{equation}\label{linearized}
|\nabla \overline{u}|+\vec{d}(\m{x})\cdot(\nabla u_1-\nabla\overline{u})=f,
\end{equation}
where $\dd(\m{x}):=\frac{\m{x}-\m{x}_0}{|\m{x}-\m{x}_0|}$. We further denote
$\dd^\perp(\m{x}):=\frac{-(x_2-x_2^0, x_1-x_1^0)}{|\x-\x_0|}$. 
Introduce the characteristic of \eqref{linearized}: 
$$\Gamma_\theta=\left\{\m{x} = (x_1,x_2):\frac{x_2-x_2^0}{x_1-x_1^0}=\tan\theta\right\},\; $$
where $\m{x}_0=(x^0_1,x^0_2)$. Denoting $p:=u_1-\overline{u}$ in  \eqref{linearized}, we obtain that for $\m{x}=(x_1,x_2)$ on the measurement surface $\Gamma$,
\begin{equation}\label{eq:fan}
p(\m{x})=\int_{\Gamma_{\theta_\m{x}}}(f-1)\,d\gamma,
\end{equation}
where $\theta_\m{x}$ satisfies $\tan\theta_\m{x} = \frac{x_2-x_2^0}{x_1-x_1^0}$.
It can be observed that  $p(\m{x})$ defines the fanbeam transform
of the function $f-1$. Therefore, the boundary  measurement of the solution $u_1$ to the linearization of Eikonal equation at the solution $\overline{u}$ can be formulated as the fanbeam transform of the inhomogeneity. 
Note that the approximation \eqref{linearized} is effective and accurate
if the contrast $|f-1|$ is small.

Although  \eqref{eq:fan}  is a convenient formula to reconstruct $f$, one can only collect the measurement of the Eikonal solution $u$ instead of the solution $u_1$ to the linearized formulation \eqref{linearized}. The formula \eqref{eq:fan} motivates us to consider the difference between the solution $u$ to the  Eikonal equation and the solution to the corresponding Fanbeam transform problem.
Assume that there exists a fanbeam solution $v$ corresponding to the slowness function $f$, i.e., $v$ solves
$$  \dd\cdot \nabla v=f.$$
We shall analyze the difference between the viscosity solution $u$ to the Eikonal solution and the fanbeam solution $v$.
It follows the definition that
\begin{equation*}
\dd\cdot \nabla (v-\overline{u})=f-1,
\end{equation*}
thus
\begin{equation}\label{eq:p1}
\dd\cdot \nabla (v-u)=f-1+\dd\cdot \nabla (\overline{u}-u):=p_1.
\end{equation}
 The following theorem provides a useful estimate of the difference between the viscosity solution $u$ and $\overline{u}$. 

\begin{theorem*} 
The viscosity solution $u$ to the Eikonal equation \eqref{eq:Eikonal} with  $f\in W^{1,\infty}(\Omega)$ satisfies
$$
|\nabla u-\nabla \overline{u}|^2=|f-1|^2+(H(\nabla u-\nabla\overline{u}),\nabla u-\nabla\overline{u}),
$$
where $H\in \mathbb{R}^{2\times2}\ge 0$ is defined by
\begin{equation} \label{H}
\begin{array}{l}
H=\left(\begin{array}{cc} \xi_2^2 &-\xi_1\xi_2
\\ \\ 
-\xi_1\xi_2 & \xi_1^2\end{array}\right), \  \xi=(\xi_1,\xi_2)=t\nabla{u}+(1-t)\nabla \overline{u}\;
\mbox{ for some }0\le t\le1.
\end{array} \end{equation}

\end{theorem*}
\begin{proof} 
We first note that given $f\in W^{1,\infty}(\Omega)$, there exists a unique viscosity solution
$u\in W^{1,\infty}(\Omega)$  to the Eikonal equation, and the solution map is continuous.
There holds that 
\begin{equation}\label{eq:diffu}
|\nabla u|-|\nabla \overline{u}|=\vec{d}\cdot(\nabla u-\nabla \overline{u})+\pp
\end{equation}
and there exists some $0\le t\le1$ such that
\begin{equation}\label{eq:ppest}
 \pp =\frac{1}{2}\,(H(\nabla u-\nabla\overline{u}),\nabla u-\nabla\overline{u})\ge0,
\end{equation}
with $H\in \mathbb{R}^{2\times2}$ defined by
\eqref{H}.
As $u$ is the solution to the Eikonal equation, we also have
$$f-1 =|\nabla u|-|\nabla \overline{u}|. $$
Together with \eqref{eq:diffu}, one obtains 
\begin{equation} \label{F}
\vec{d} \cdot \nabla u=f-\pp.
\end{equation}
It is noted that 
\begin{equation*}
|\nabla u-\nabla \overline{u}|^2= |\nabla u|^2-|\nabla \overline{u}|^2-2\vec{d}\cdot(\nabla u-\nabla \overline{u}),
\end{equation*}
and 
$$|f|^2-1=|f-1|^2+2(f-1).$$
Together with \eqref{eq:diffu} we have the estimate 
\begin{equation*}
|\nabla u-\nabla \overline{u}|^2= |f-1|^2+(H(\nabla u-\nabla\overline{u}),\nabla u-\nabla\overline{u}).
\end{equation*}
\end{proof}
Note that the term $(H(\nabla u-\nabla\overline{u}),\nabla u-\nabla\overline{u})$ defined with \eqref{H}  is equal to  $|\dd^\perp\cdot(\nabla u-\nabla\overline{u})|^2$
when $t=0$ (at $\nabla \overline{u}$).  If we assume that $|f-1|$ is sufficiently small, then
$|\nabla u-\nabla \overline{u}|^2-(H(\nabla u-\nabla\overline{u}),\nabla u-\nabla\overline{u})$ can be approximated by 
$ |\dd\cdot(\nabla u-\nabla \overline{u})|^2$.
Thus the following corollary on  $\nabla u-\nabla \overline{u}$  and the function $p_1$ follows the theorem:

\begin{corollary}\label{cor:1} Assume that \begin{equation} \label{est}
|\nabla u-\nabla \overline{u}|^2-(H(\nabla u-\nabla\overline{u}),\nabla u-\nabla\overline{u})\ge
(1-\delta)^2\,|\dd\cdot(\nabla u-\nabla \overline{u})|^2,
\end{equation}
then
$$
|\nabla u-\nabla \overline{u}|^2\le |f-1|^2+\delta\, (f-1) \text{ and } \pp \le \delta\,( f-1).
$$
\end{corollary}
One can further deduce the $L^1$ estimate of $\pp\ge0$ by integrating along the characteristic $
\Gamma_\theta$,
$$
\int_\Gamma (u-\overline{u})d\gamma=\int_\theta \int_{\Gamma_\theta}(f-1+\pp)\,d\gamma\,d\theta.
$$
It then follows that 
$$
\Vert p_1\Vert_{L^1(\Omega)}=\int_\theta \int_{\Gamma_\theta}\pp\,d\gamma\,d\theta=\int_\Gamma (u-\overline{u})\,d\gamma
-\int_\theta \int_{\Gamma_\theta}(f-1)\,d\gamma\,d\theta.
$$
  


Recall that the  fanbeam solution $v$ satisfies 
$$
\dd\cdot \nabla (v-u)=\pp,
$$
which leads to
\begin{equation}\label{eq:diff}
\Vert v-u\Vert_{L^1(\Omega)}\le C\,\Vert \pp\Vert_{L^1(\Omega)}
\end{equation}
for some constant $C$. Under the assumption in Corollary \ref{cor:1}, the estimate \eqref{eq:diff} indicates that the fanbeam solution $v$ approximates the Eikonal solution $u$ well when $f-1$ is sufficiently small.

Using formula \eqref{F}, now we propose a  two-step procedure for the fanbeam-based reconstruction method of $f$. 
As the fanbeam inverse problem is a better-posed problem compared to the Eikonal inverse problem that  is ill-posed as discussed
in Subsection \ref{sub:illposed},
we propose to 'transform' measurement from the Eikonal inverse problem to the fanbeam problem by a two-step approach, and we will further develop a direct filtered-back projection method in Section \ref{sec:fbp} for the fanbeam inverse problem   within this approach.
%
The approach consists of two steps: first, based on \eqref{F}, we apply the inverse fanbeam transform with measurement of  Eikonal solution $u^k$ at $\Gamma$ corresponding to source $\m{x}_0^k$ to obtain an estimate $\hat{f}$ of $f$; next, we solve the Eikonal equation with the corresponding point source condition and  $\hat{f}$ to derive an approximation $\hat{v}^k$ of the Eikonal solution $u^k$, and plug $\hat{v}^k$ into \eqref{eq:p1} to compute $\hat{p}^k$. Summing $\hat{p}^k$ over $k$ and adding this correction term to $\hat{f}$, we arrive at  the approximated slowness function.
This process can be formulated as Algorithm \ref{fb0}.  \vspace{3mm}
\begin{algorithm} 
\caption{Algorithm for inverse Eikonal problem}\label{fb0}
\begin{algorithmic}[1]
\State For each point source $\m{x}^k_0$ $(k = 1, 2, ..., m)$, denote the Eikonal solution corresponding to the unknown slowness function $f$ by $u^k$ 
and collect measurement $p^k: = u^k |_\Gamma$  on the surface $\Gamma: = \partial\Omega$. Compute the Eikonal solution  $\overline{u}^k =|\m{x}-\m{x}^k_0| $ corresponding to the constant background slowness function,  and denote its value on the boundary $\Gamma$ by $\overline{p}_0^k$. Compute $\vec{d}^k_0 = \frac{\nabla \overline{u}^k}{|\nabla \overline{u}^k|}$. \

\State Apply the inverse Fanbeam transform with data $p^k$ to compute $\hat{f}$.
\State Solve the Eikonal equation 
$$|\nabla \hat{v}^k|=\hat{f}, \quad \hat{v}^k(\m{x}^k) = 0$$ 
for $\hat{v}^k$.
\State Let $\hat{p}^k=|\nabla \hat{v}^k|-1-\dd_0^k\cdot(\nabla \hat{v}^k-\nabla\overline{u}^k)$ and compute approximated slowness function $$f = \hat{f}+\sum_{k=1}^m \hat{p}^k. $$

\end{algorithmic}
\end{algorithm}



\section{Radon transform and filtered back projection method}\label{sec:fbp}
In this section, we introduce a filtered back projection (FBP) method for the inverse Radon transform, which motivates the algorithm for the inverse fanbeam transform in the next section. For a function $f$ contained in a compact set $\Omega\subset \mathbb{R}^2$, the Radon transform $R$ of the function $f$ is given by
\begin{equation*}
Rf(\m{\alpha},t):=\int_{\m{x}\cdot \m{\alpha} = t}f(\m{x})dx_L = \int^\infty_{-\infty} f(t \alpha_1-u\alpha_2,t\alpha_2+u\alpha_1)\,du,
\end{equation*}
where $\alpha = (\alpha_1,\alpha_2)\in S^1$, $\m{x}:=(x_1,x_2)\in\mathbb{R}^2$, and $t = \m{x}\cdot \alpha$ represents a hyperplane with normal direction $\alpha$ and distance $t$ to the origin.
The adjoint transform $R^*$ defined on $g\in L^\infty( S^1, \mathbb{R}^1)$ is given by
\begin{equation*}
R^*g(x_1,x_2)=\int_{S^1} g(\alpha, \m{x}\cdot\alpha)\,d\alpha.
\end{equation*}
One can compute the inverse Radon transform $R^{-1}$ with the following formula \cite{helgason1980radon}:
$$
R^{-1}p=R^*{\cal H}\partial_t p(\alpha,t),
$$
where ${\cal H}$ denotes the Hilbert transform, i.e., for a function $F$, the Hilbert transform of $F$ can be defined explicitly as
$$
{\cal H}F:=\lim_{\epsilon\to0} \int_\epsilon^\infty \frac{F(t+ p)-F(t- p)}{ 2p}dp.
$$
Denoting the Fourier transform by  $\cal F$, one obtains the filter step  given by
$$
\Phi p={\cal F}^{-1}(|\nu|{\cal F})={\cal H}\partial_t p(\alpha,t)=(-\Delta_t)^{1/2}p(\alpha,t)
$$
for each angle and the Laplacian $\Delta_t$ in $t$, where  $\nu$  is the
variable in the frequency domain  and the fractional Laplacian is defined as in \cite{kwasnicki2017ten}.
Now we propose  the filtered inverse Radon transform defined by
\begin{equation} \label{fb}
f=R^*S\,p(\alpha,t),
\end{equation}
where the scaling filter $S$ is defined by
\begin{equation} \label{ope:S}
S=(-\Delta_t)^{1/2}(c-\Delta_t)^{-1},
\end{equation} 
and  $c>0$ denotes the regularization parameter which is selected according to the noise level.

It is noted that this regularized filtered back projection method can be extended to the general case. Consider a general inverse problem for determining the source $f$ from measurement $y$  governed by the equation
$$
Af=y
$$
for a closed, densely defined linear operator $A:X\to Y$, where $X$ and $Y$ denote two Hilbert spaces. Denote the range of $A$ by $R(A)$ and define the graph norm of $y\in R(A)$ by
$$
\Vert y\Vert^2_A = (y,(AA^*+cP)^{-1}y)_Y,
$$ 
where $(,)_Y$ is the natural inner product of the Hilbert space $Y$ and  $P$ is a positive self-adjoint operator. 
We can further define an inner product $(,)_A$:
$$(x,y)_A=(x,(AA^*+c\,P)^{-1}y)_Y.$$
Since
$$
(y,Ax)_A=(y,(AA^*+c\,P)^{-1}Ax)_Y,
$$
we define an space adjoint operator $A^\dagger$ of $A:X\to R(A)$ by
$$
A^\dagger=A^*(AA^*+c\,P)^{-1}.
$$
If $c=0$, the corresponding $A^\dagger$ satisfies
$$
AA^\dagger y=y,\;\; y\in R(A).
$$
Thus, $A^\dagger$ 
defines the filtered back projection operator.

Conversely, define the graph norm of $R(A^*)\subset X$ by $\Vert f\Vert_{ A^*}:= \Vert Af\Vert$ and the corresponding inner product $(\cdot,\cdot)_{A^*}$ in $R(A^*)$ is defined by 
\begin{equation*}
(f, g)_{A^*} := (Af, Ag)_Y.
\end{equation*}
Then
$$
A^\dagger=A^*(AA^*)^{-1}\mbox{  on  } R(A)
$$
defines the adjoint operator of $A:R(A^*) \to Y$, since
$$
(y,Af)_Y=(AA^\dagger y,Af)_Y = (A^\dagger y, f)_{A^*}.
$$
For the specific problem,  Radon transform, we have
$$
(-\Delta_s)^{1/2}=(AA^*)^{-1}.
$$

\section{Inverse fanbeam transform}\label{sec:4}
In this section, we extend the algorithm for the inverse Radon transform in Section \ref{sec:fbp} to the inverse fanbeam transform. Let $D_R$ denote a disk containing a bounded domain $\Omega$ and $S^1_R = \partial D_R$.  For a function $f$ contained in $\Omega\subset \mathbb{R}^2$, the fanbeam transform is defined by
\begin{equation}
\FR f(\theta,\m{x}_0)=\int_\Omega f(\m{x})\delta((\m{x}-\m{x}_0)\cdot \alpha)\,d\m{x}
=\int f(\m{x}_0+u\theta )\,du
\end{equation}
where $\delta$ denotes the dirac delta function, $\theta \in S^1$ and $\m{x}_0:=(x_{0,1}, x_{0,2})$ denotes the point source distributed on $S^1_R$. Then the adjoint transform is given by
\begin{equation} \label{adj}
(\FR^*p)(\m{x})=\int_{S_R^1} p(\arctan(\frac{x_2-x_{0,2}}{x_1-x_{0,1}}),\m{x}_0)\,d\m{x}_0.
\end{equation}

To extend the back projections algorithms for the Radon transform to the fanbeam transform, we let $T$ be the coordinate transform from the Radon transform to the Fanbeam transform, see \cite{zeng2010medical} for this procedure of extension with different algorithms.
With the coordinate transform operator, we deduce that
$$
\FR \FR^* = (TR)(TR)^*=TRR^*T^* 
$$
is block diagonal, since $RR^*$ is  anglewise $(-\Delta_s)^{-1/2}$  diagonal.  Then we propose the direct probing method for the inverse fan beam transform
based on \eqref{fb}--\eqref{ope:S}:
\begin{equation} \label{bf0}
f= \FR^*S\,p,
\end{equation}
where $p$ denotes the measurement of fanbeam transform, $\FR^*$ denotes the adjoint transform defined by \eqref{adj}, and $S$ denotes the scaling filter defined in \eqref{ope:S}.

\begin{remark}
\begin{itemize}
\item When only limited-angle measurement is available,
one can interpolate the limited-angle sinogram data by a (periodic cubic) spline in $\alpha$, then the inverse fanbeam method \eqref{bf0} can be applied.
\item The algorithm for the inverse Fanbeam problem can be extended to other problems with fanbeam geometry. 
\end{itemize}
\end{remark}
\section{Assumed background and adjoint based back projection}\label{sec:5}

The inverse fanbeam transform is an efficient approximation of Eikonal tomography when the slowness distribution has a small variation from the homogeneous background. For the high-contrast media, since the linearization approach is no longer accurate, we extend our algorithm by utilizing priori information of the high-contrast background in this section.  

Assume that the unknown slowness distribution $f\in W^{1,\infty}(D_R)$ contained in $\Omega$ is close to a given background $\overline{f}\in W^{1,\infty}(D_R)$, and let $\overline{u}\in W^{1,\infty}(D_R)$ be the viscosity solution to
$$
|\nabla \overline{u}|-\overline{f}=0,\ \overline{u}(\m{x}_0) = 0.
$$
The linearized equation at $\overline{u}$ of the Eikonal equation \eqref{eq:Eikonal} is
\begin{equation}\label{eq:linearinhomo}
\vec{d}_0\cdot \nabla( u_1-\overline{u})=f-\overline{f},
\end{equation}
where $\vec{d}_0=\frac{\nabla \overline{u}}{|\nabla \overline{u}|}$.
Denoting the operator on the lefthand side of \eqref{eq:linearinhomo} by 
\begin{equation}\label{eq:E}
E(u_1-\overline{u}):=\vec{d}_0 \cdot \nabla (u_1-\overline{u}),
\end{equation}
one can rewrite the linearized equation as $$
E(u_1-\overline{u})=f-\overline{f}.
$$
Consider the value of solutions on the boundary $\Gamma$,  $p_1: = u_1|_\Gamma$ and  $\overline{p}:=\overline{u}|_\Gamma$. There holds that  
$$
p_1-\overline{p}=T_\Gamma E^{-1}(f-\overline{f}):=A(f-\overline{f}),
$$
where  $T_\Gamma$ denotes the trace operator from $W^{1,\infty}(D_R)$ to $C(\Gamma)$. The adjoint operator $A^*$ of $A$  can be computed by
$$
A^*=(T_\Gamma E^{-1})^*=(E^{*})^{-1}T_\Gamma^*,
$$
then  $\lambda:=A^*(p_1-\overline{p}) = (E^{*})^{-1}T_\Gamma^*(p_1-\overline{p})$ is an approximation of  $f-\overline{f}$. 
It follows \eqref{eq:E} that
$$
(Eu,\lambda)=(\vec{d}_0 \cdot \nabla u ,\lambda)=-(\nabla\cdot (\vec{d}_0\lambda),u)+(n\cdot \vec{d}_0 \lambda,u)_\Gamma,
$$
where $(\cdot,\cdot)_\Gamma$ denotes the $L^2$ inner product on $\Gamma$ and $n$ denotes the outer normal direction on $\Gamma$.
By the definition of $\lambda$, one can also deduce
$$
(Eu,\lambda)=(u ,E^*\lambda)=(u, T_\Gamma^*(p_1-\overline{p}))= (u, p_1-\overline{p})_\Gamma.
$$
These two equations above indicate that
\begin{equation}\label{eq:adjoint}
\nabla\cdot (\vec{d}_0\,\lambda)=0,  \;\; n\cdot \vec{d}_0\lambda|_\Gamma=p_1-\overline{p}.
\end{equation}
When the difference between $f$ and $\overline{f}$ is small,  the linearized solution $p_1$ is also close to the measurement $p: = u|_\Gamma$ of the Eikonal solution $u$, then the filtered back projection algorithm \ref{fb0} can be extended to the inverse Eikonal problem corresponding to a high contrast medium by solving the adjoint equations. The extended approach is presented as Algorithm  \ref{Covsolve}.

%
%
Note that when solving the adjoint equation \eqref{eq:adjoint}, one can multiply $\phi$ on both sides to deduce the variational formulation \begin{equation} \label{adit-d}
(-\nabla\cdot (\vec{d}_0\lambda),\phi)=(\lambda,\vec{d}_0\cdot \nabla\phi)-(n\cdot\vec{d}_0\lambda,\phi)_\Gamma=0.
\end{equation}
If we take $\phi$ satisfying $\vec{d}_0\cdot\nabla\phi=\lambda$ in this formulation, we obtain
$$
\Vert \lambda\Vert_{L^2(\Omega)}^2=(n\cdot \vec{d}_0\lambda,\phi)_\Gamma.  
$$
Thus, \eqref{adit-d} admits a weak solution $\lambda\in L^2(\Omega)$.
Also, one can define the solution as
$$
\lambda=\exp(\int (\nabla\cdot \vec{d}_0)\,dt)q,
$$
where we assume $\vec{d}$ is Lipschitz and $\vec{x}(t)$ is the backward characteristic curve defined by ODE:
$$
\frac{d}{dt}x(t)=\vec{d}(x(t)),\quad x(T)=x\in \Gamma.
$$
Then  the well-poshness  of \eqref{eq:adjoint} is shown if
$\vec{d}_0$ is sufficiently smooth, say  $\vec{d}_0\in (W^{1,\infty}(\Omega))^2$.  In general, $\vec{d}_0$  can be very singular, 
thus in our proposed algorithm  we first apply a Gaussian filter to regularize
$\vec{d}_0$ and then consider the viscous dual equation
$$
\nabla\cdot (\vec{d}_0 \cdot\nabla\lambda)=\varepsilon \Delta\lambda,
$$
which corresponds to the viscous Eikonal equation
$$
-|\nabla u|+f +\varepsilon\Delta u=0.
$$
\begin{algorithm}[H]
\caption{Algorithm for inverse Eikonal problem with high-contrast  background}\label{Covsolve}
\begin{algorithmic}[1]
\State For each point source $\m{x}_0^k$ $(k=1,2,...,m)$, 
measure the first arrival time (i.e. solution $u^k$ to the Eikonal equation) on the boundary $\Gamma$ corresponding to the unknown slowness $f$, and denote the measurement as $p^k$.  Given the background slowness function $\overline{f}$ (which may be high contrast profile),  compute the Eikonal solution  $\overline{u}^k$  corresponding to $\overline{f}$ and denote its value on the boundary $\Gamma$ by $\overline{p}^k$. Compute $\vec{d}^k_0 = \frac{\nabla \overline{u}^k}{|\nabla \overline{u}^k|}$.

\State Solve the following equation for each $k$ to deduce $\lambda_k$:
 \begin{equation}
 \begin{aligned}
\epsilon \Delta \lambda^k - \nabla\cdot (\lambda^k \vec{d}^k_\alpha) = 0\quad \text{ in } \Omega,\\
n\cdot\vec{d}^k_\alpha \lambda_k= p^k-\overline{p}^k\quad \text{ on } \Gamma,
 \end{aligned}
 \end{equation}
 where $\vec{d}^k_\alpha$ is deduced by applying Gaussian filter to $\vec{d}^k_0$.
\State Sum $\lambda_k$ over $k$ as the reconstruction of $ f-\overline{f}$.
\end{algorithmic}
\end{algorithm}

\section{Numerical findings and discussions}\label{sec:6}

In this section, we present several numerical findings and carry out a series of implementations to illustrate the robustness and efficiency of the proposed algorithms. In the following examples, the slowness function $f$ is supporte in a square domain $\Omega_2 = [-0.5, 0.5]\times [-0.5, 0.5]$ contained in the circular domain $\Omega = B(0,0.75)$. For each velocity model, the boundary measurements corresponding to several point sources respectively are collected on the boundary $\Gamma$ of the circular domain $\Omega$ for the reconstruction. Both the set of point sources and the set of measured points are equally distributed on $\Gamma$. We shall call these measurements the Eikonal sinogram as an analogue of the sinogram for the Radon transform. The synthetic boundary measurements are computed with the fast switching method with mesh size $h=0.01$. The noisy measurements are generated by adding a stationary additive Gaussian random noise  to the exact boundary measurements:
\begin{equation}
p_s(\theta,x_0) = p_e(\theta, x_0)+\varepsilon \cdot \text{max}_\theta(p_e)\cdot\xi,  
\end{equation}
where $p_e$ denotes the exact data,  $\xi$ follows the standard normal distribution, and $\varepsilon$ denotes the relative noise level.
\subsection{Numerical findings}
We shall illustrate the ill-posedness of the Eikonal inverse tomography numerically in this subsection. In Example \ref{exp:1}-\ref{exp:3}, the measurement corresponding to $18$ sources are collected at 153 points on the boundary.
\begin{example}\label{exp:1}
   Consider the velocity model \eqref {illposedex1} in Section \ref{sec:2}. We compare the measurements corresponding to the velocity models with different magnitudes in this example, that is, we measure the solution of the Eikonal equation on the boundary with different $f_0$ in the velocity model  \eqref {illposedex1}.
   \end{example}
The experiments verify that the measured sinograms with $f_0 = 1.5$ and $f_0 = 2$ are the same. Thus we numerically verify that the Eikonal tomography is severely ill-posed as the same measurements are collected on the boundary for certain scenarios regardless of the value of contrast.
\begin{figure}[ht!]
\centering
\subfigure[Slowness function $f$ with $f_0 = 1.1$]{
\label{d01}
\includegraphics[width=0.31\linewidth]{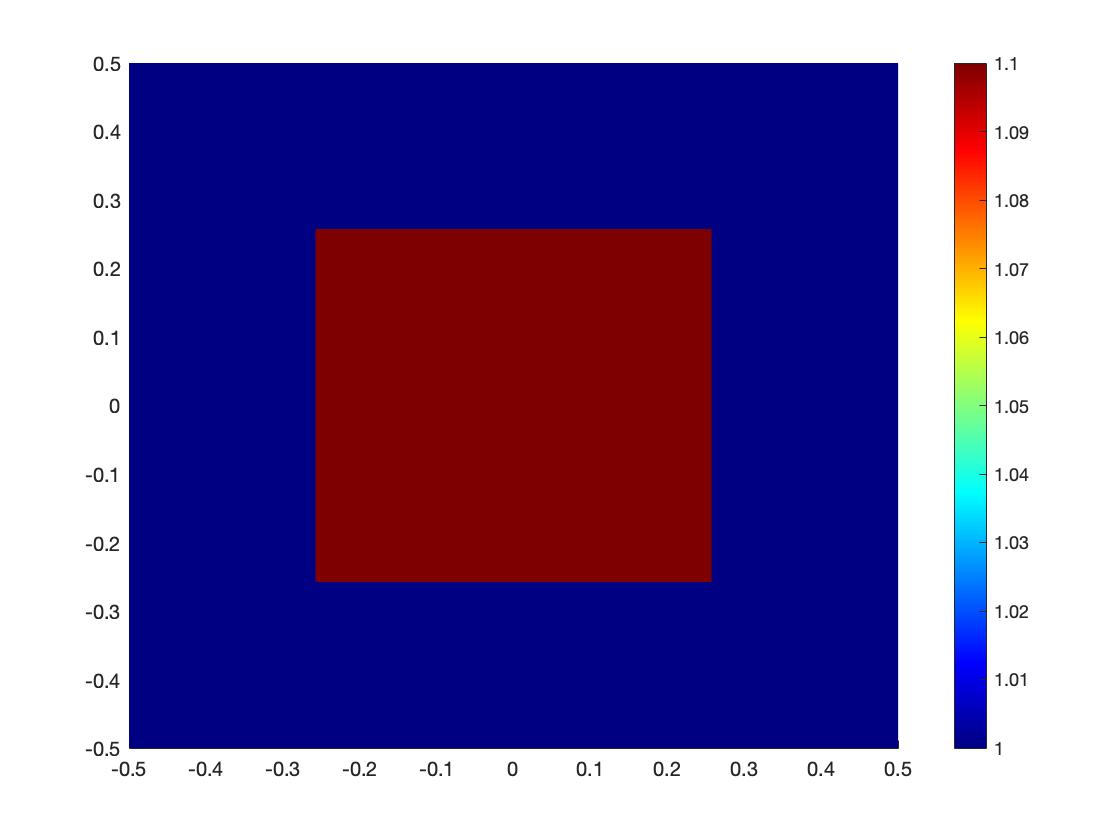}}
\subfigure[Eikonal sinogram]{
\label{d02}
\includegraphics[width=0.31\linewidth]{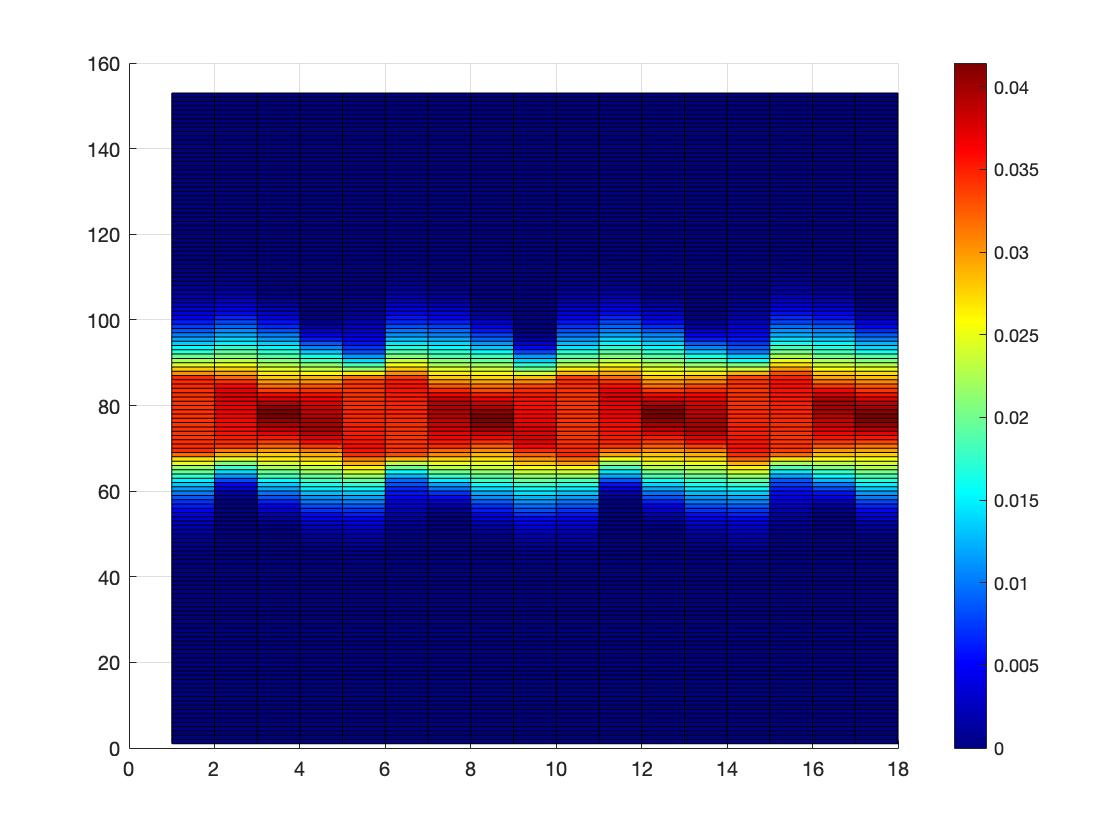}}
\subfigure[FBP reconstruction]{
\label{d03}
\includegraphics[width=0.31\linewidth]{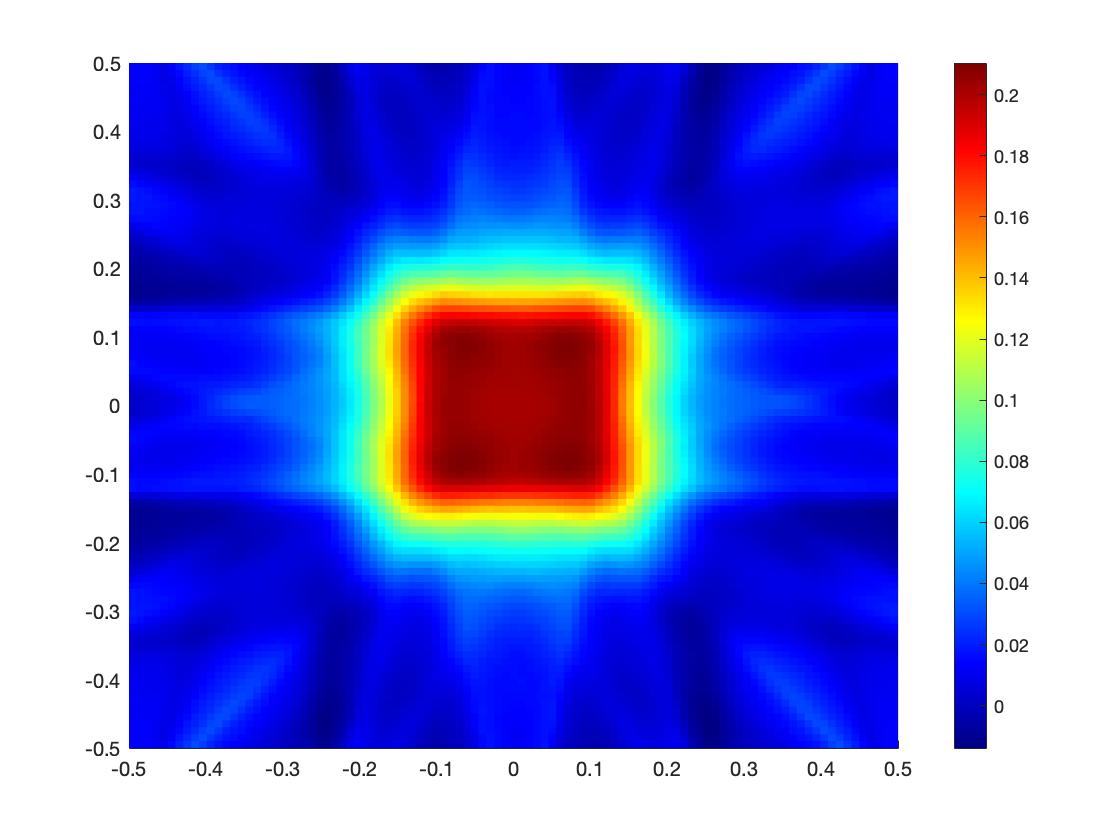}}
\subfigure[Slowness function $f$ with $f_0 = 1.5$]{
\label{d4}
\includegraphics[width=0.31\linewidth]{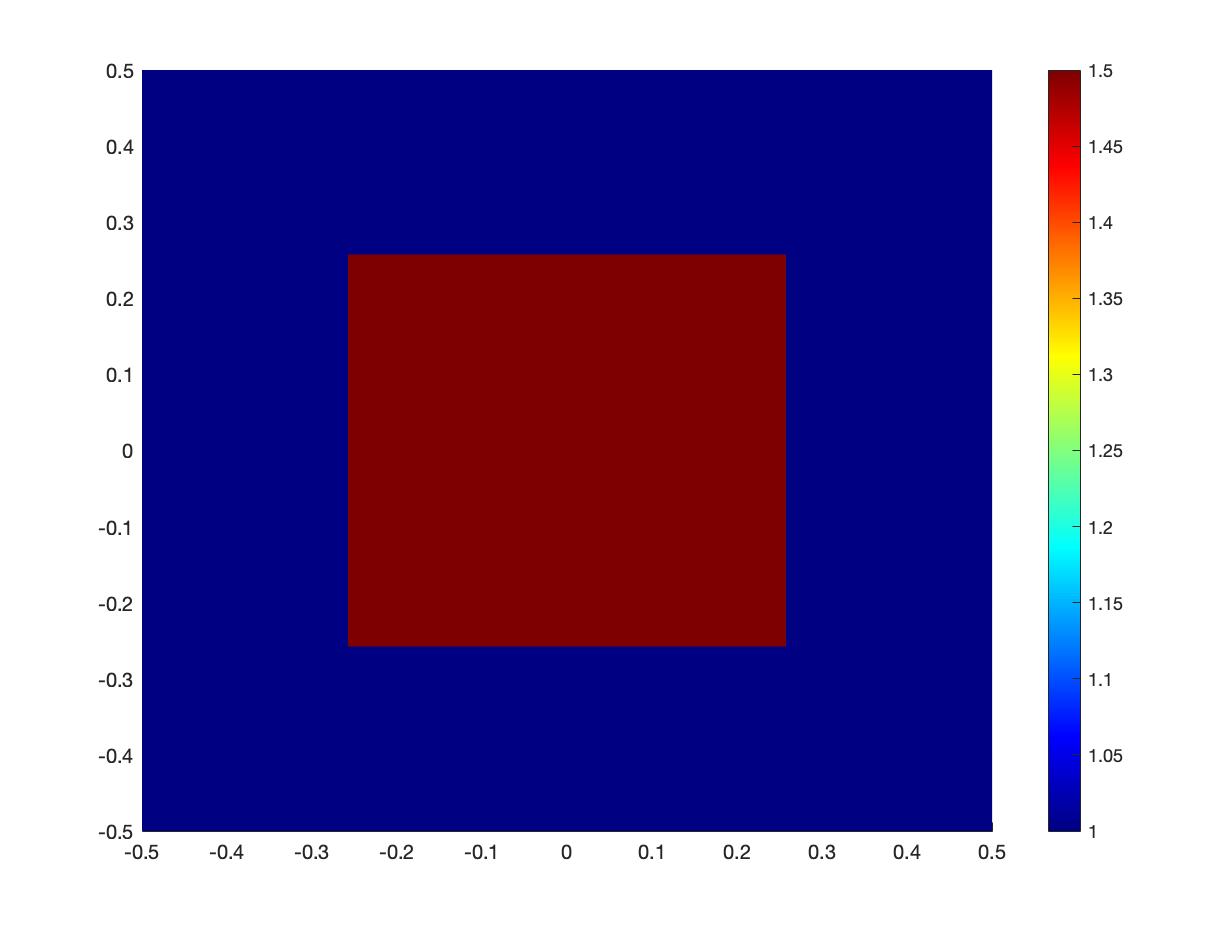}}
\subfigure[Eikonal sinogram]{
\label{d5}
\includegraphics[width=0.31\linewidth]{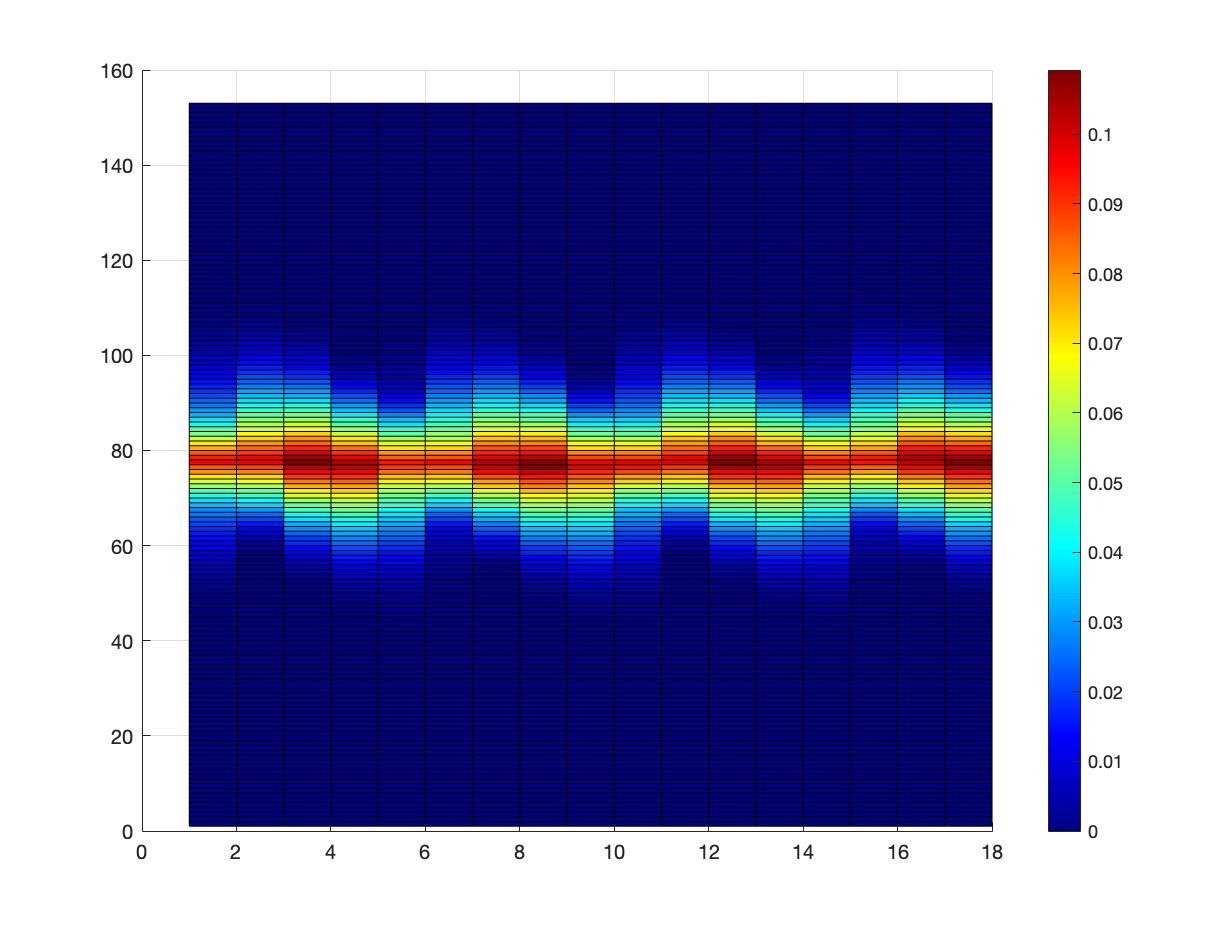}}
\subfigure[FBP reconstruction]{
\label{d6}
\includegraphics[width=0.31\linewidth]{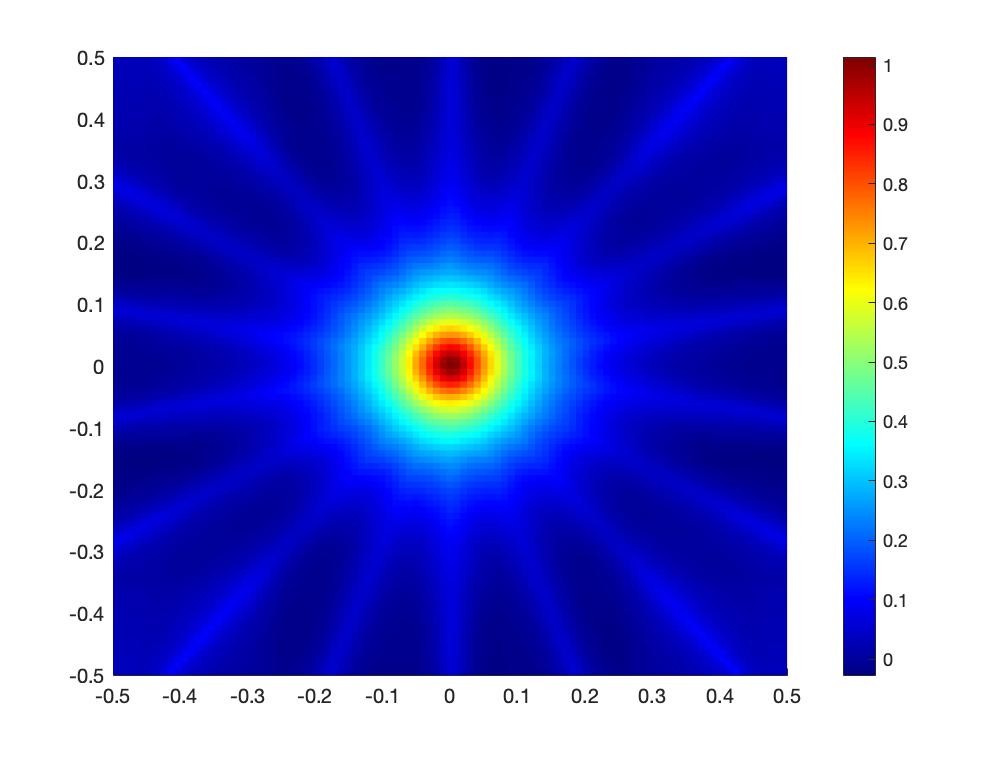}}
\subfigure[Slowness function $f$ with $f_0 = 2$]{
\label{d4}
\includegraphics[width=0.31\linewidth]{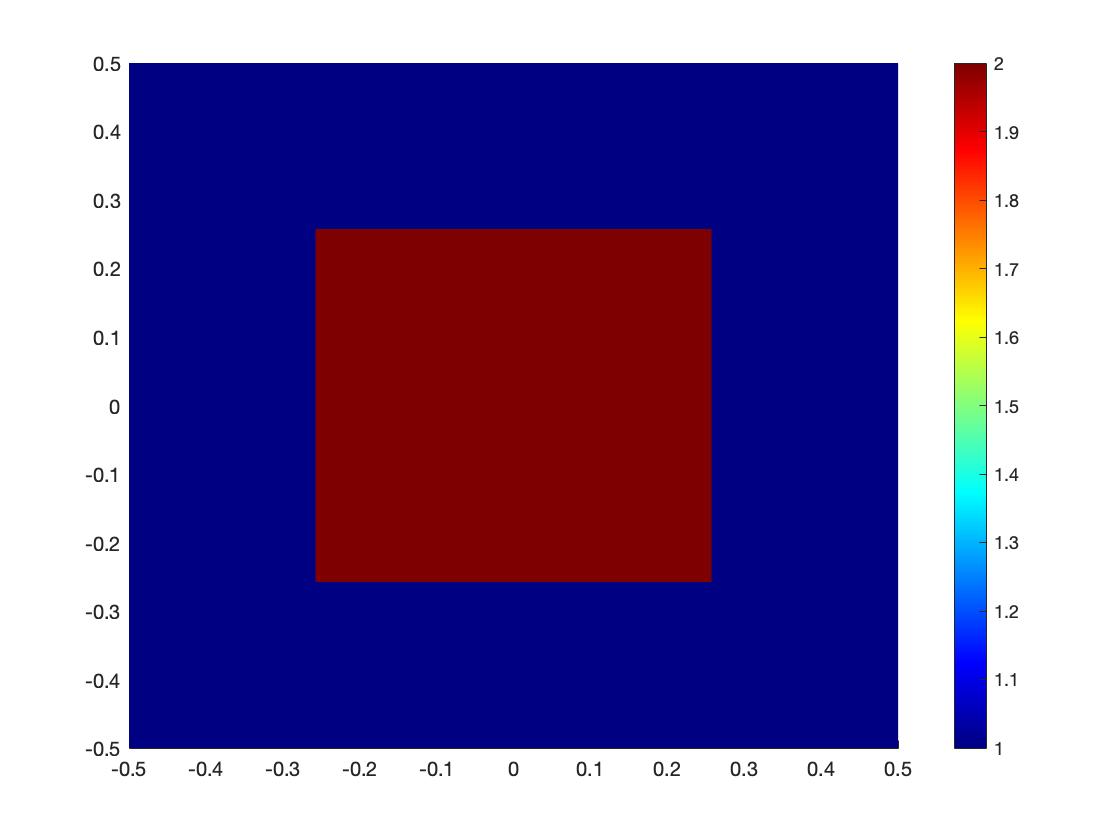}}
\subfigure[Eikonal sinogram]{
\label{d5}
\includegraphics[width=0.31\linewidth]{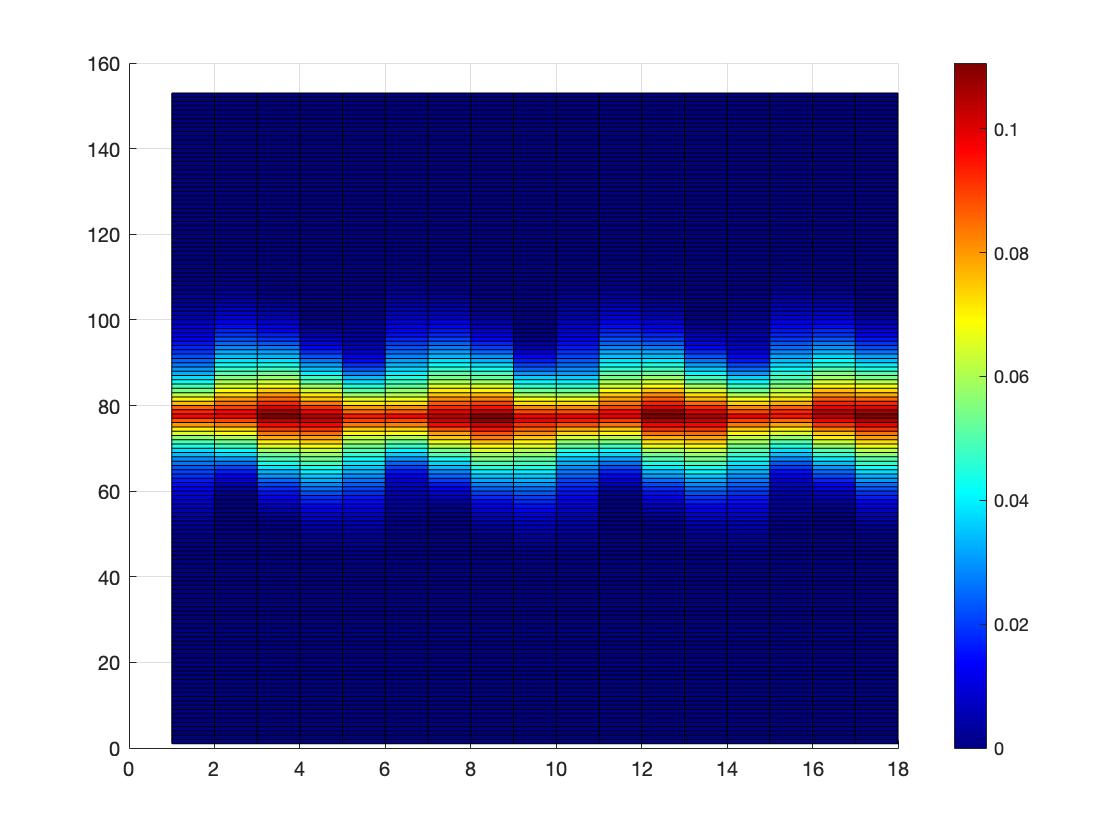}}
\subfigure[FBP reconstruction]{
\label{d6}
\includegraphics[width=0.31\linewidth]{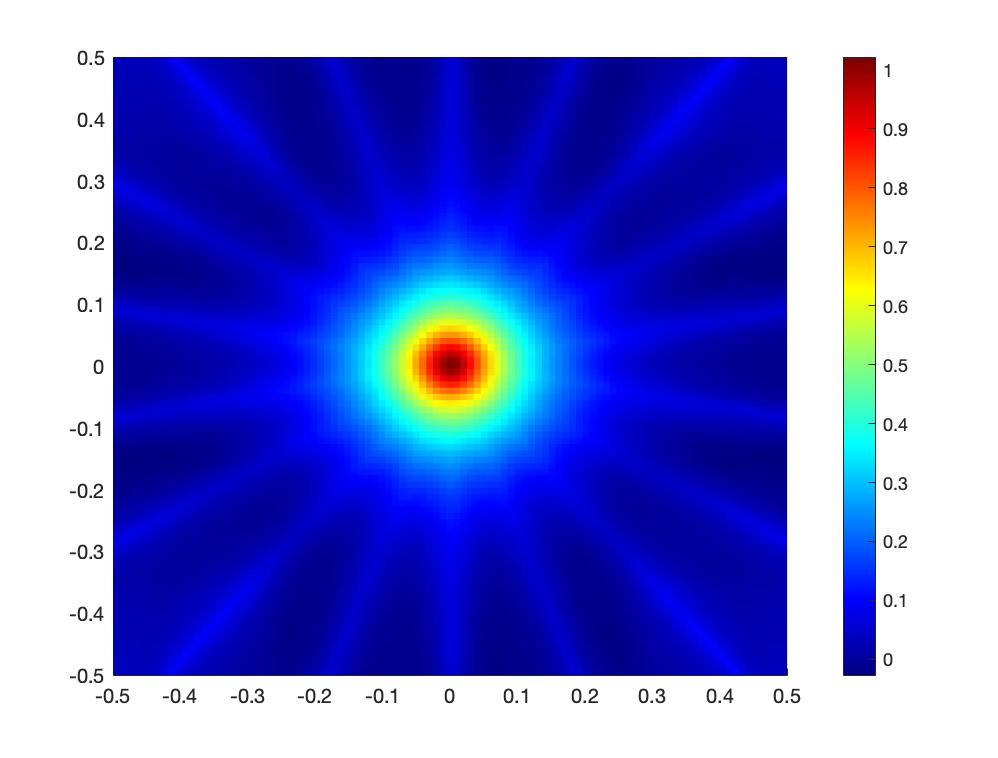}}
\caption{Example \ref{exp:1}. Comparison of the sinograms and reconstructions corresponding to the velocity models with different magnitudes. }
\end{figure}

\begin{example}\label{exp:2}
In this example, we present the fanbeam sinogram and the FBP reconstruction for the velocity model with two inclusions of size $0.2\times 0.2$ with different contrast, located respectively at positions $(-0.20,-0.20)$ and $(0.20,-0.10)$. 
\end{example}

Although the exact slowness inside the inclusion centered at $(0.20,-0.10)$ is larger than the other one, the difference is not correctly reflected in the reconstruction. This implies that only the profile of the velocity can be reconstructed for the high-contrast medium due to the ill-posed nature of the problem. 

\begin{figure}[ht!]
\centering
\subfigure[Slowness function $f$]{
\label{qq01}
\includegraphics[width=0.31\linewidth]{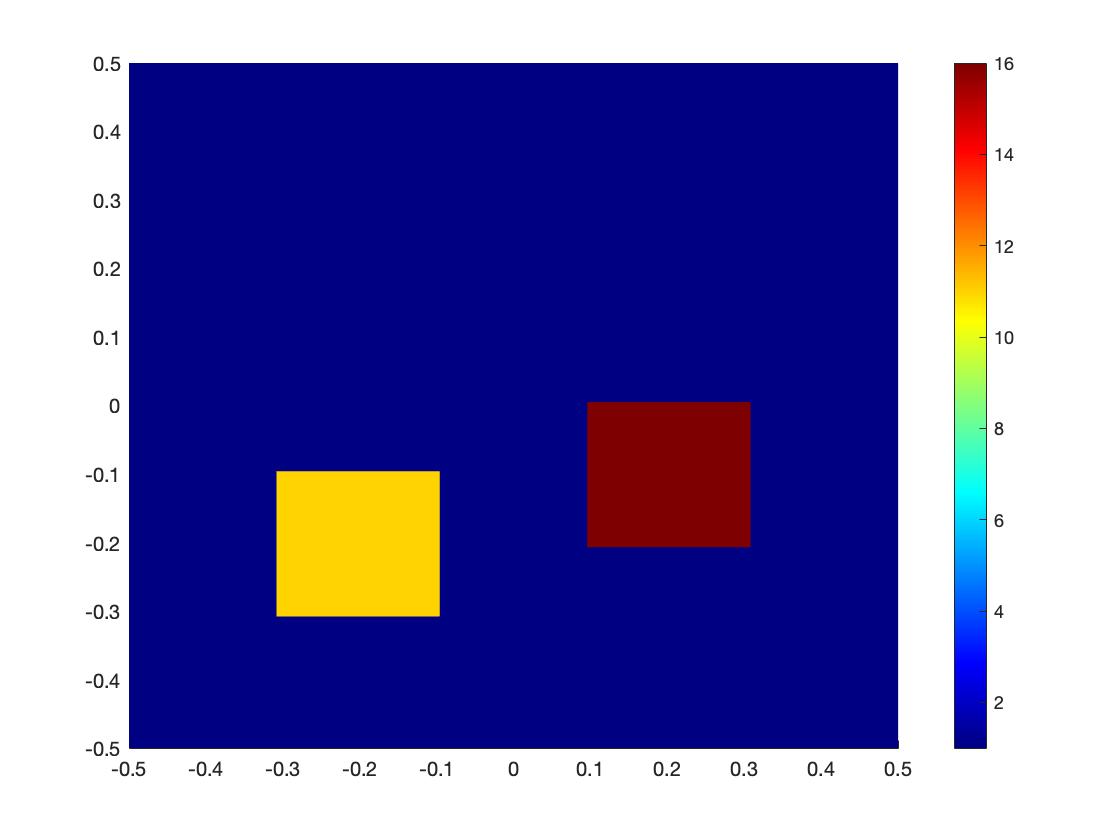}}
\subfigure[Eikonal sinogram]{
\label{qq02}
\includegraphics[width=0.31\linewidth]{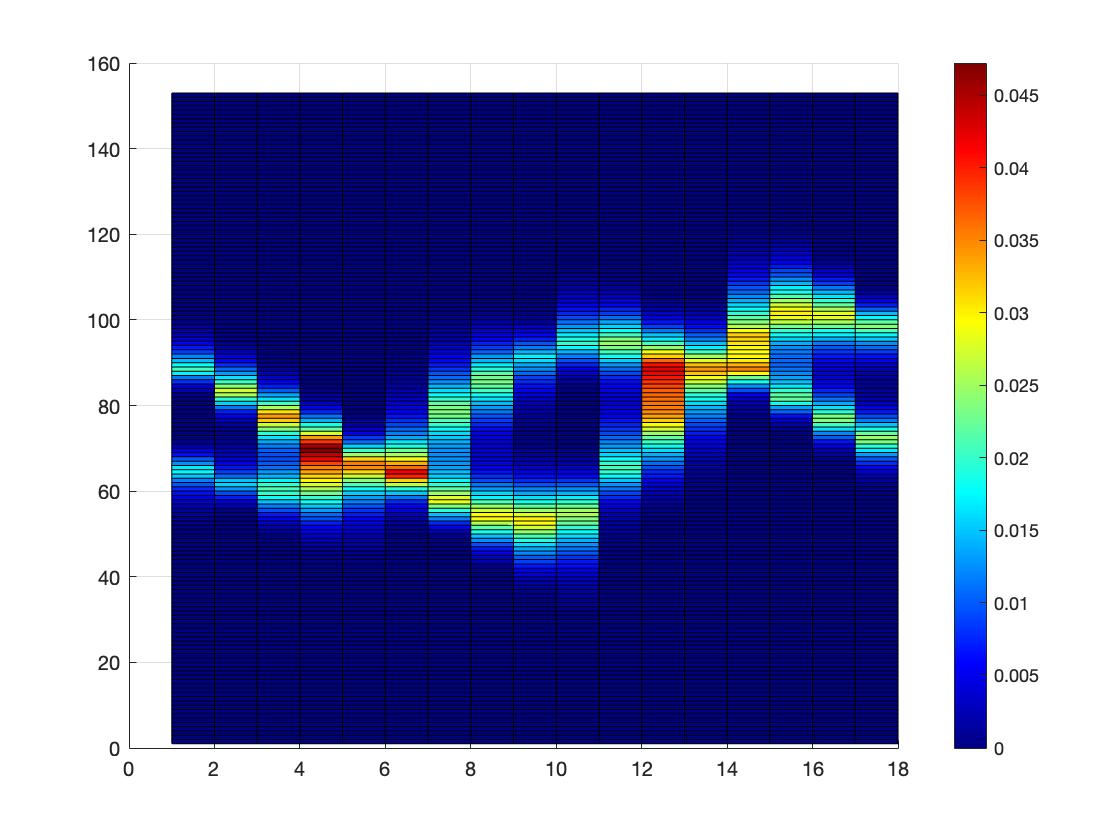}}
\subfigure[FBP reconstruction]{
\label{qq03}
\includegraphics[width=0.31\linewidth]{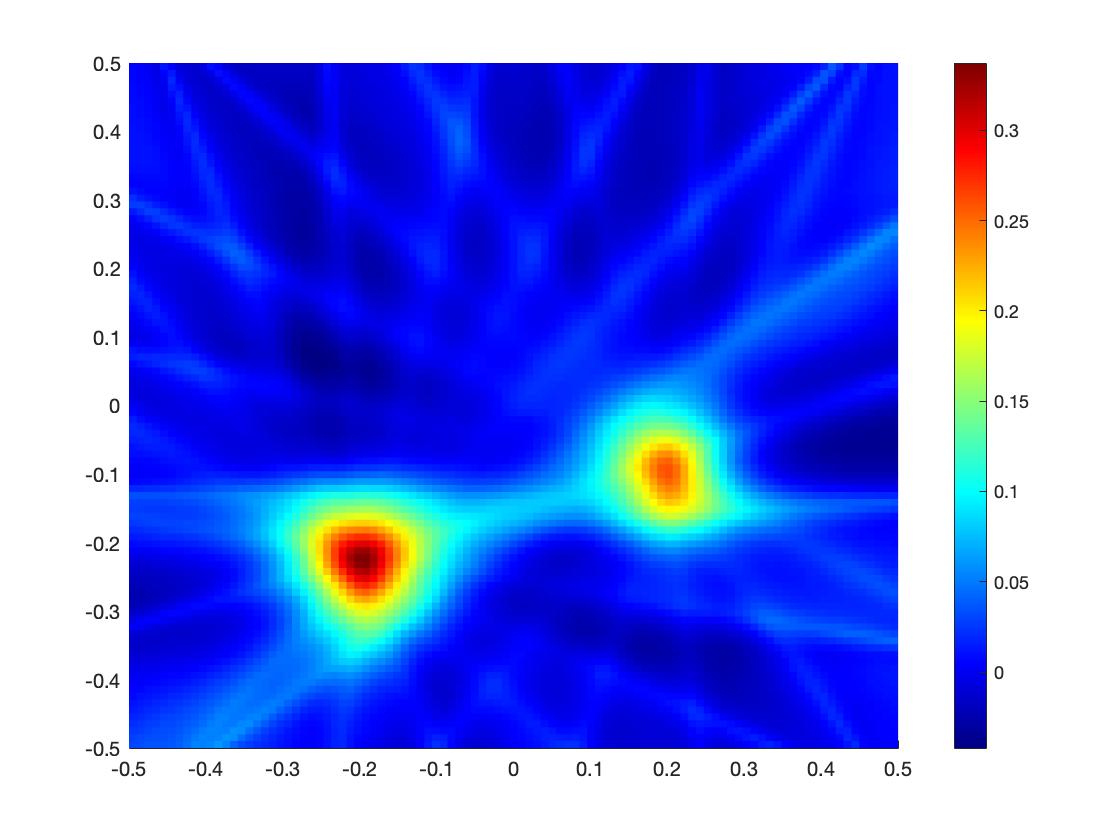}}
\caption{Example \ref{exp:2}. Reconstruction of velocity model with two inclusions of different contrast. }
\end{figure}

\begin{example}\label{exp:3}
Now we present the fanbeam sinogram and the Eikonal sinogram for the same velocity model with four small inclusions of size $0.1\times 0.1$, centered at  $(-0.25,-0.25)$, $(0.30,-0.35)$, $(0.25, 0.35) $, and $(-0.20,0.20)$  as shown in \ref{d4}. The slowness inside the inclusions is $1.5$ and the background slowness is $1$.
\end{example}
 In previous sections we consider the inverse fanbeam transform as an efficient approximation of Eikonal tomography when the slowness distribution is close to the homogeneous background. The fanbeam sinogram can be considered as the line integrals of wave amplitude attenuation for straight ray-paths connecting the point sources and the measurement surface, while for the velocity model with inhomogeneity,  the ray-paths can depend strongly on the unknown wave speeds, and thus the resulting Eikonal sinogram displays similar but different patterns to fanbeam sinogram, as shown in \ref{d5} and \ref{d6}. Therefore it is necessary to filter and refine the back projection methods as in our proposed algorithm. This velocity model will be revisited in example 4 to examine the performance of the reconstruction algorithm.

\begin{figure}[ht!]
\centering
\subfigure[Slowness function $f$]{
\label{d4}
\includegraphics[width=0.32\linewidth]{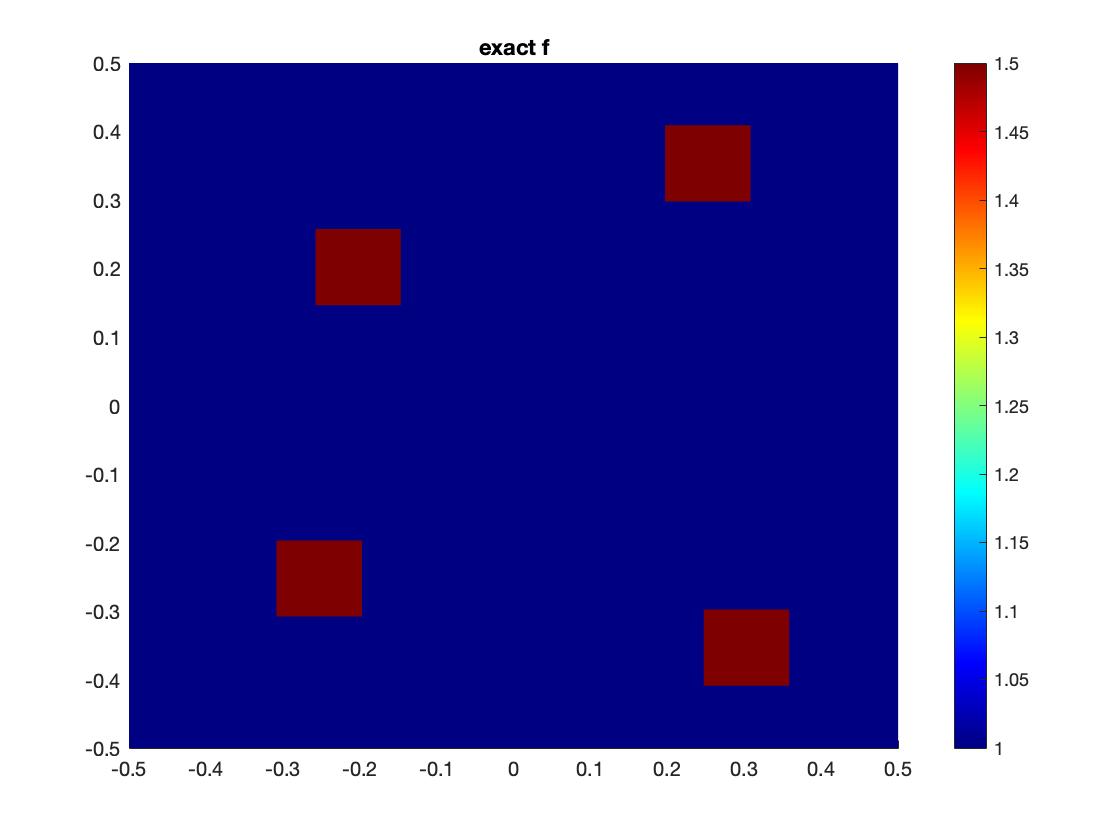}}
\subfigure[Fanbeam sinogram]{
\label{d5}
\includegraphics[width=0.32\linewidth]{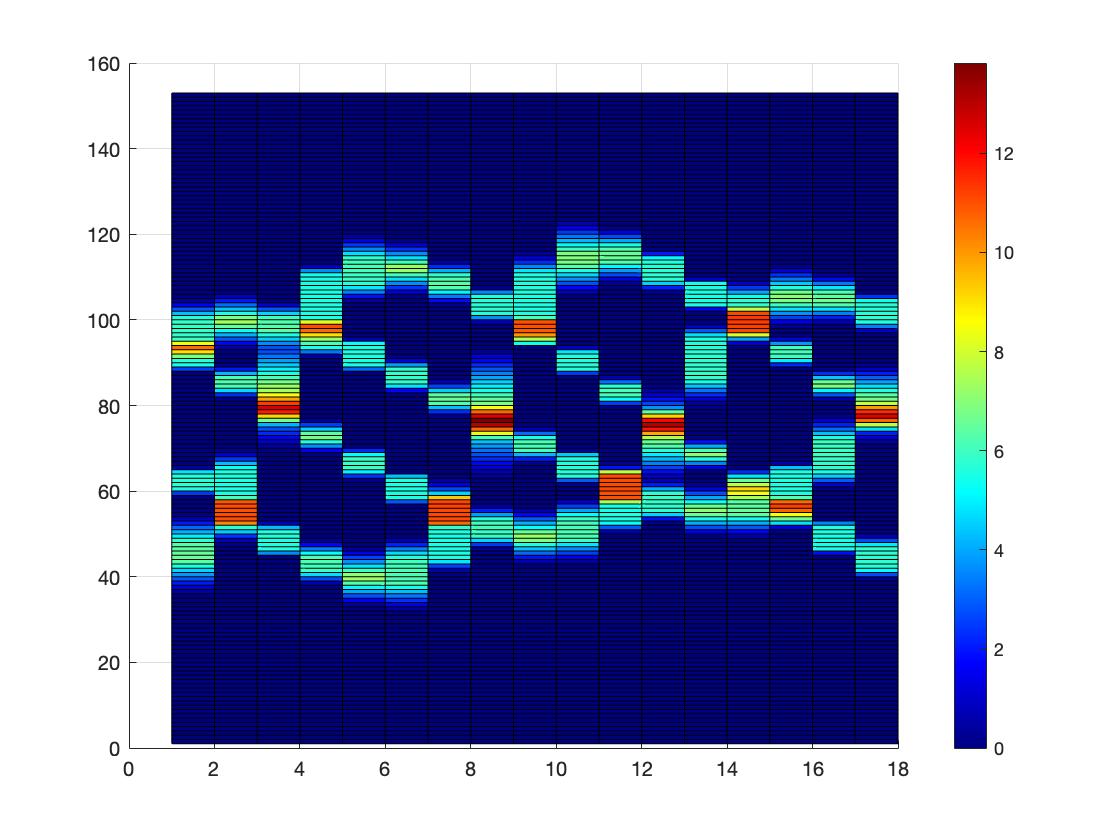}}
\subfigure[Eikonal sinogram]{
\label{d6}
\includegraphics[width=0.32\linewidth]{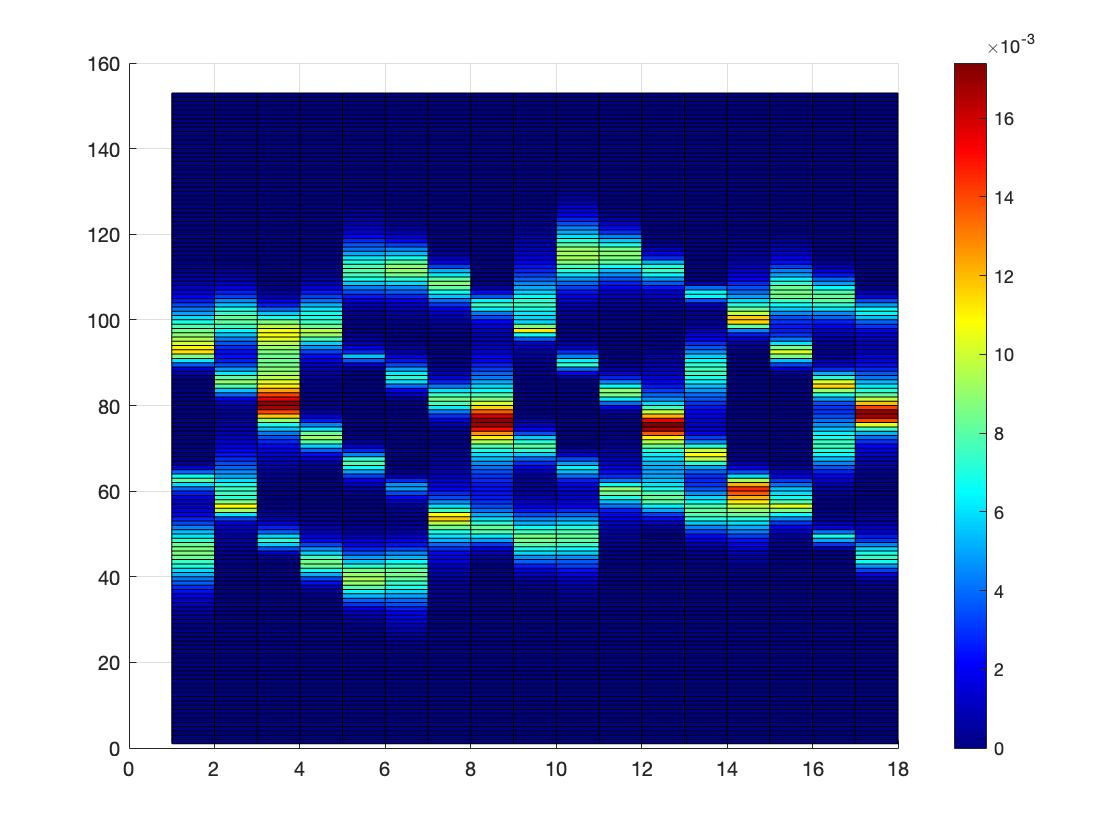}}
\caption{Example \ref{exp:3}. Comparison of the fanbeam sinogram and the Eikonal measurement. }
\end{figure}

\subsection{Two-step FBP method}
Now we implement the algorithm \ref{fb0} on two velocity models to demonstrate the efficiency and
accuracy of the proposed two-step direct probing method.  The measurement corresponding to $18$ sources is collected at 153 points on the boundary for the reconstruction in Example \ref{exp:4}, and the measurement corresponding to $36$ sources collected at 153 points is required for the reconstruction in Example \ref{exp:5}.

%
%
%

\begin{example}\label{exp:4}
We consider the velocity model in Example \ref{exp:3}, which contains four inclusions in the homogeneous background with the slowness $f = 1.5$ inside the inclusions and $f = 1$ in the background.
\end{example}

It is observed that although the Eikonal sinogram and the fanbeam sinogram displays different patterns,  the reconstruction with the FBP method from the Eikonal sinogram can provide a quite accurate indicator of the locations of these four inclusions.  The overall profile stands out clearly and agrees well with the exact velocity model. The refinement step does not significantly improve the reconstruction when the size of inclusions is relatively small. With the presence of $5\%$ noise in measurement, the FBP method still leads to satisfying reconstruction.  Hence the proposed algorithm is tolerant with respect to data noise.

\begin{figure}[ht!]
\centering
\subfigure[Slowness function $f$]{
\label{d1}
\includegraphics[width=0.31\linewidth]{Figures/far_exact.jpg}}
\subfigure[FBP with exact measurement]{
\label{d2}
\includegraphics[width=0.31\linewidth]{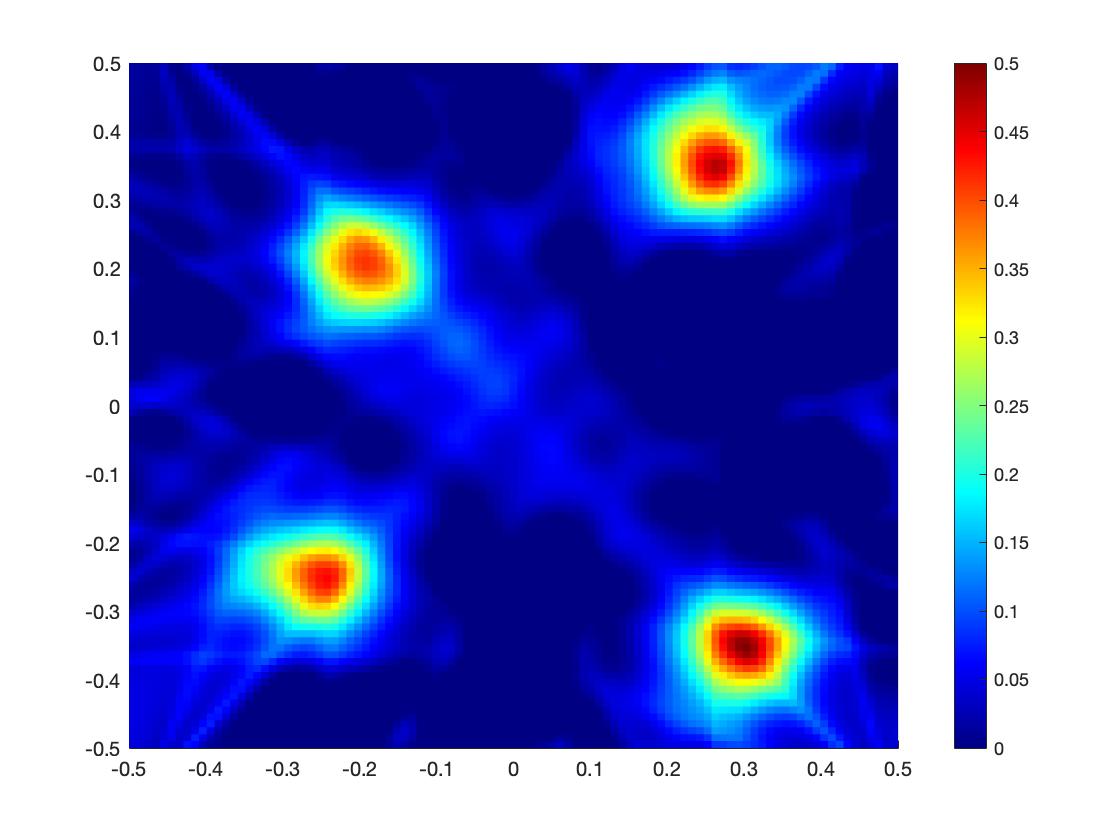}}
\subfigure[Reconstruction after refinement]{
\label{d3}
\includegraphics[width=0.31\linewidth]{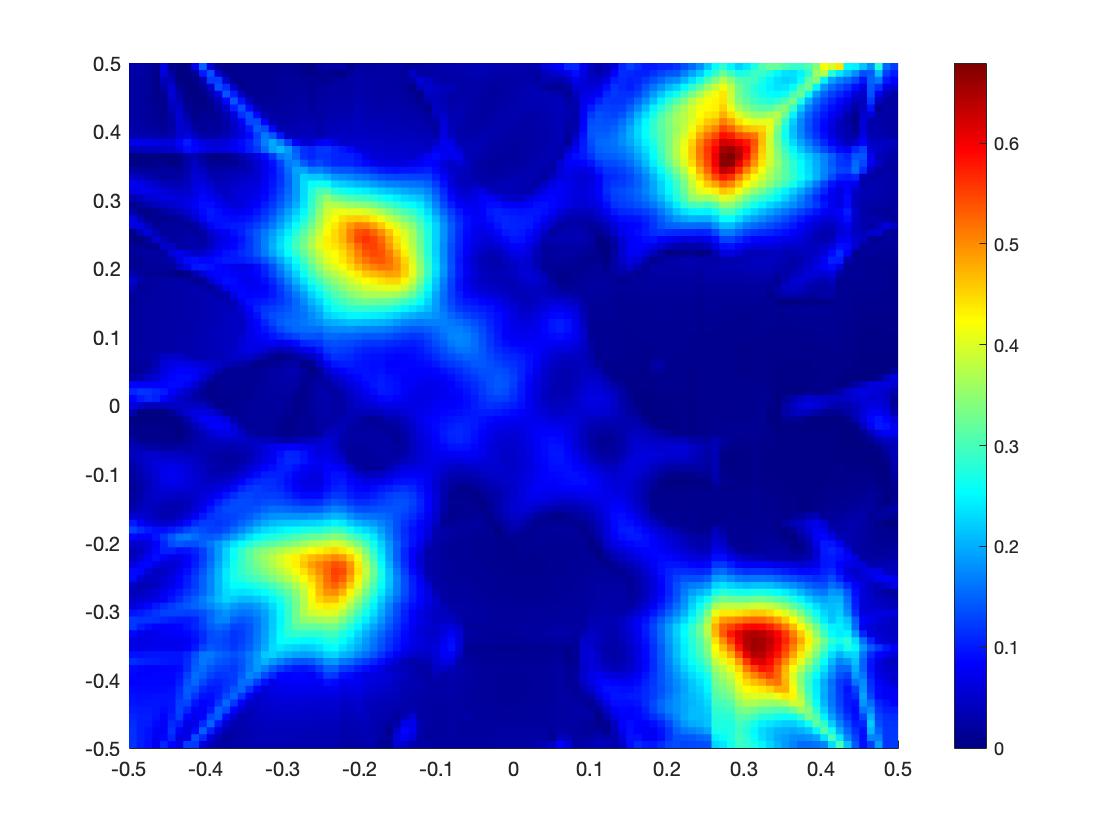}}
\subfigure[Slowness function $f$]{
\label{d4}
\includegraphics[width=0.31\linewidth]{Figures/far_exact.jpg}}
\subfigure[FBP with noisy measurement]{
\label{d5}
\includegraphics[width=0.31\linewidth]{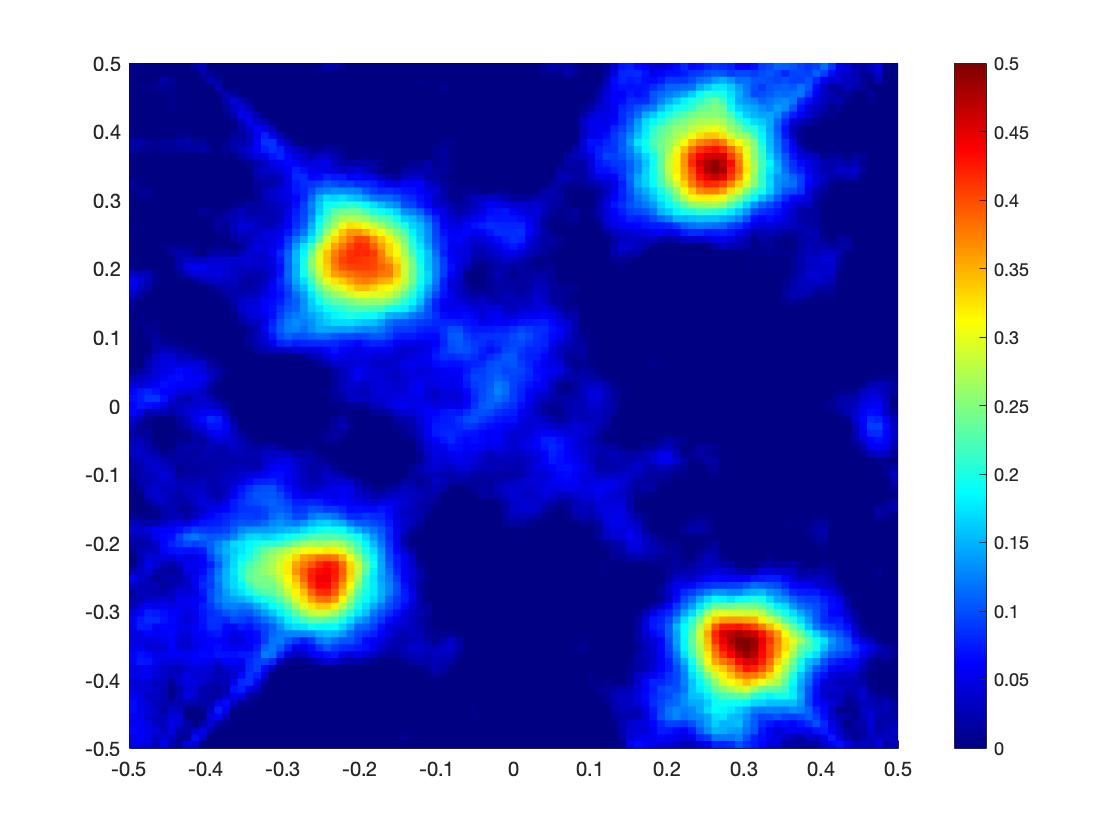}}
\subfigure[Reconstruction after refinement with noisy measurement]{
\label{d6}
\includegraphics[width=0.31\linewidth]{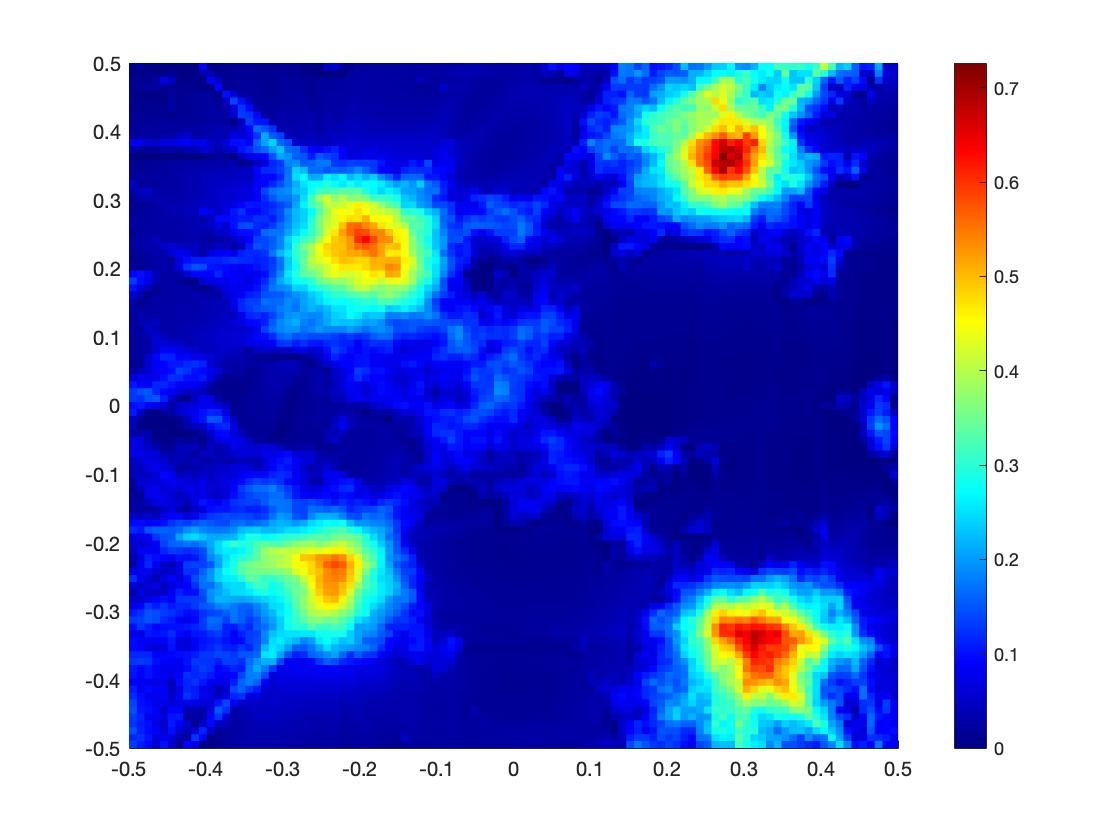}}
\caption{Example \ref{exp:4}. Reconstructions of discontinuous slowness function with exact and noisy measurement with noise level $5\%$ }
\end{figure}

\begin{example}\label{exp:5}
In this example, the velocity model with a ring-shaped square inclusion is examined as shown in Fig. \ref{e1}. The outer and inner side lengths of the ring-shaped inclusion are $0.6$ and $0.5$, and the inclusion is centered at $(0,0)$. The slowness is taken to be $f = 1.05$ inside the region and $f=1$ as the background. The reconstruction with the exact data and noisy data with $5\%$ noise are presented.
\end{example} 

 Such ring-shaped inclusions are relatively challenging to recover, yet the overall profile stands clearly in the reconstruction \ref{e2} of the FBP method.  It can be observed that the refinement step enhances the reconstruction and exhibits a clear ring structure which agrees excellently with the exact velocity model.  Our probing method remains stable with respect to noise in the data.

\begin{figure}[ht!]
\centering
\subfigure[Slowness function $f$]{
\label{e1}
\includegraphics[width=0.31\linewidth]{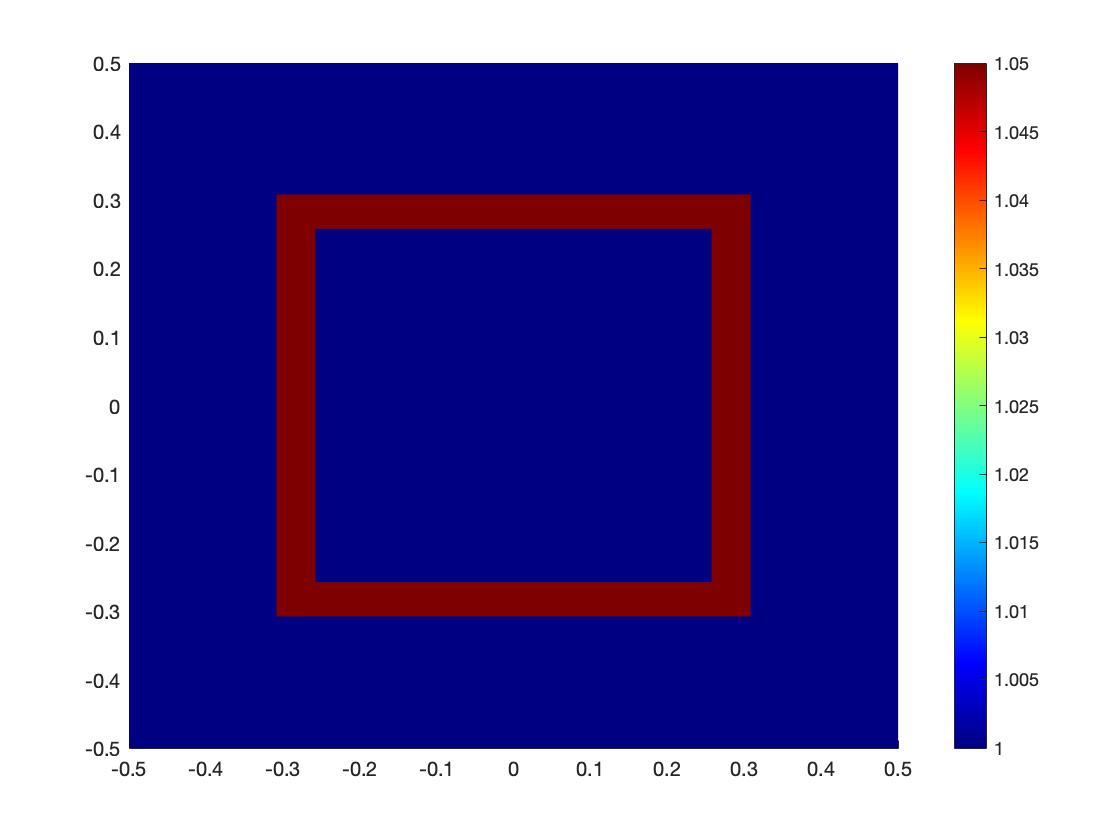}}
\subfigure[FBP with exact measurement]{
\label{e2}
\includegraphics[width=0.31\linewidth]{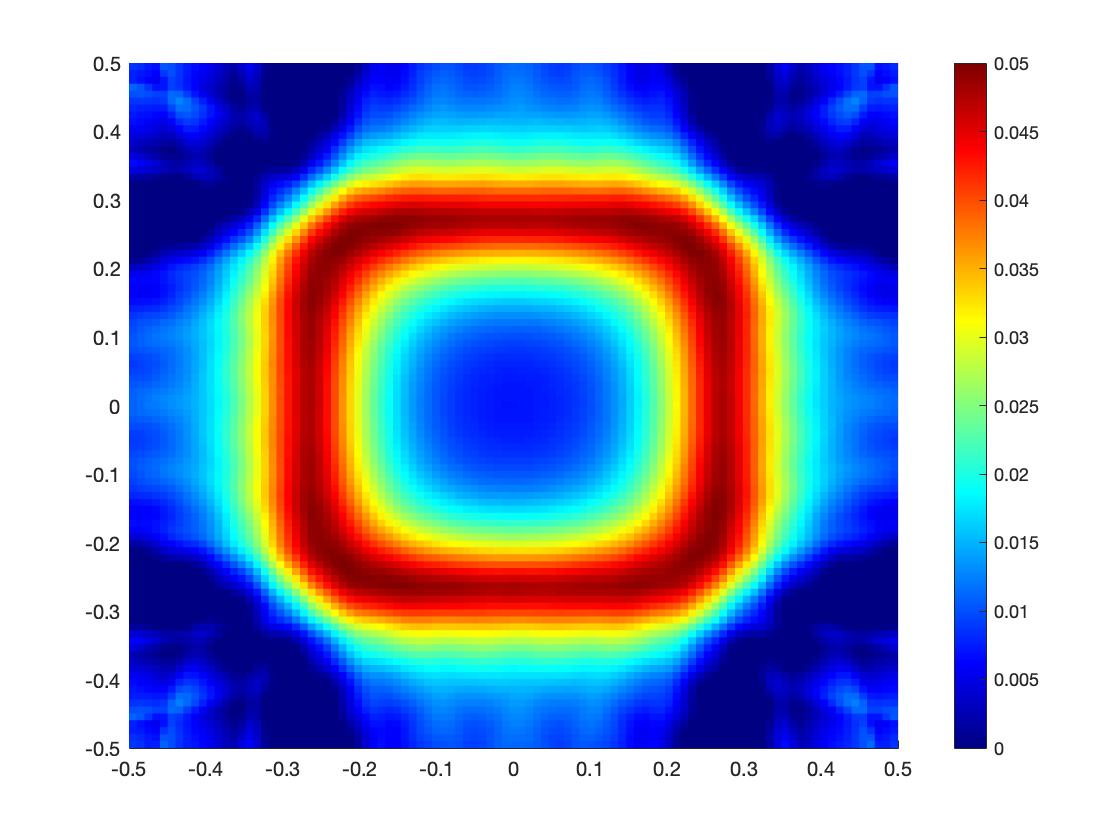}}
\subfigure[Reconstruction after refinement]{
\label{e3}
\includegraphics[width=0.31\linewidth]{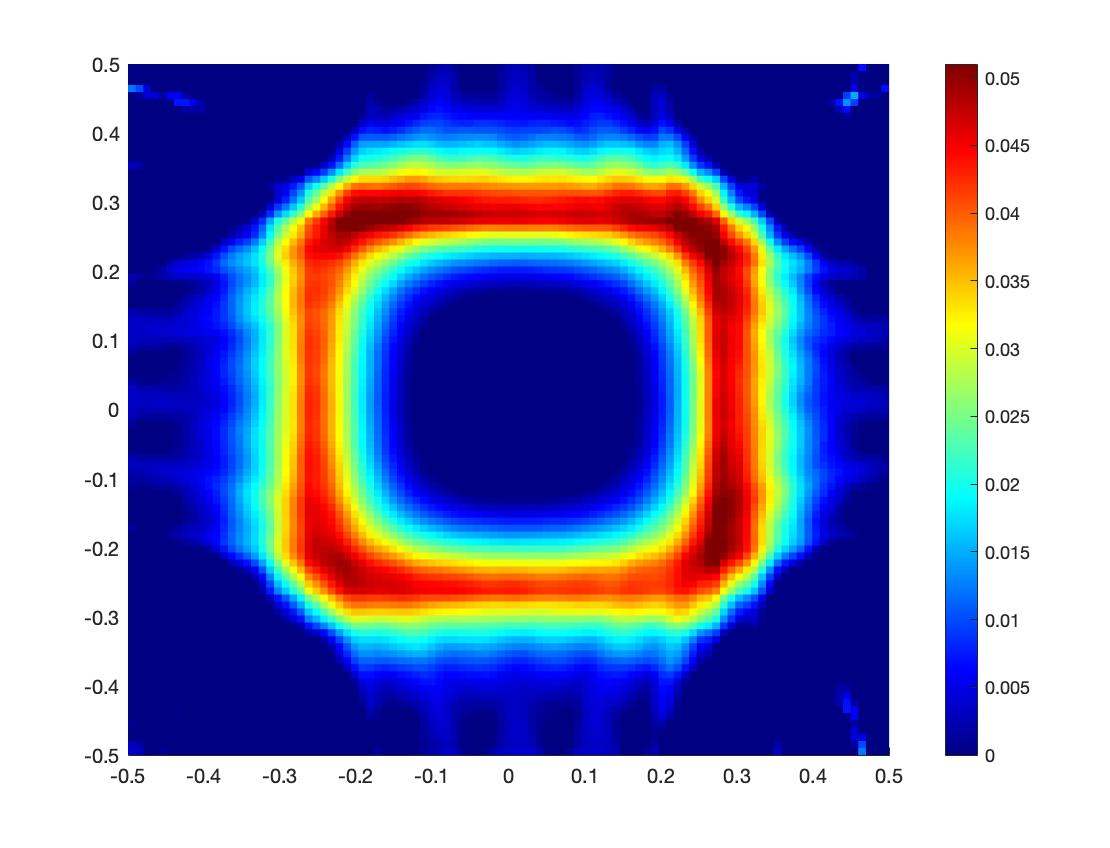}}
\subfigure[Slowness function $f$]{
\label{e4}
\includegraphics[width=0.31\linewidth]{Figures/recring_exact.jpg}}
\subfigure[FBP with noisy measurement]{
\label{e5}
\includegraphics[width=0.31\linewidth]{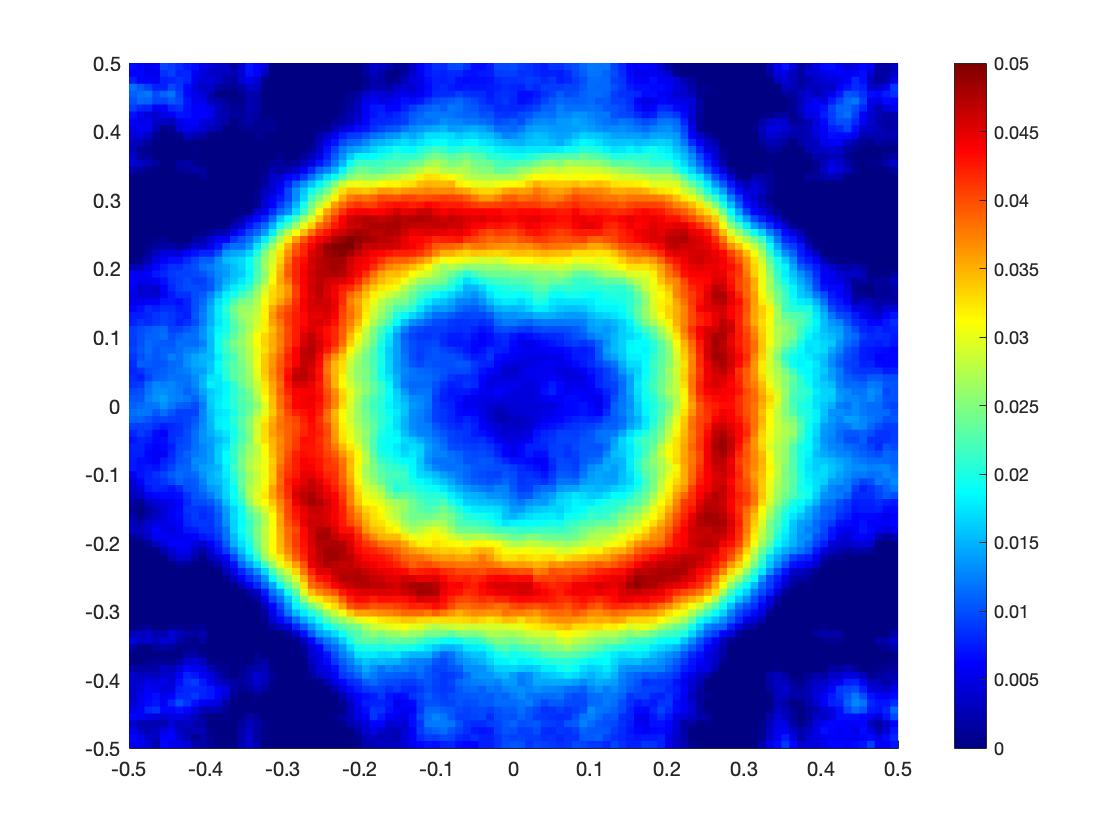}}
\subfigure[Reconstruction after refinement]{
\label{e6}
\includegraphics[width=0.31\linewidth]{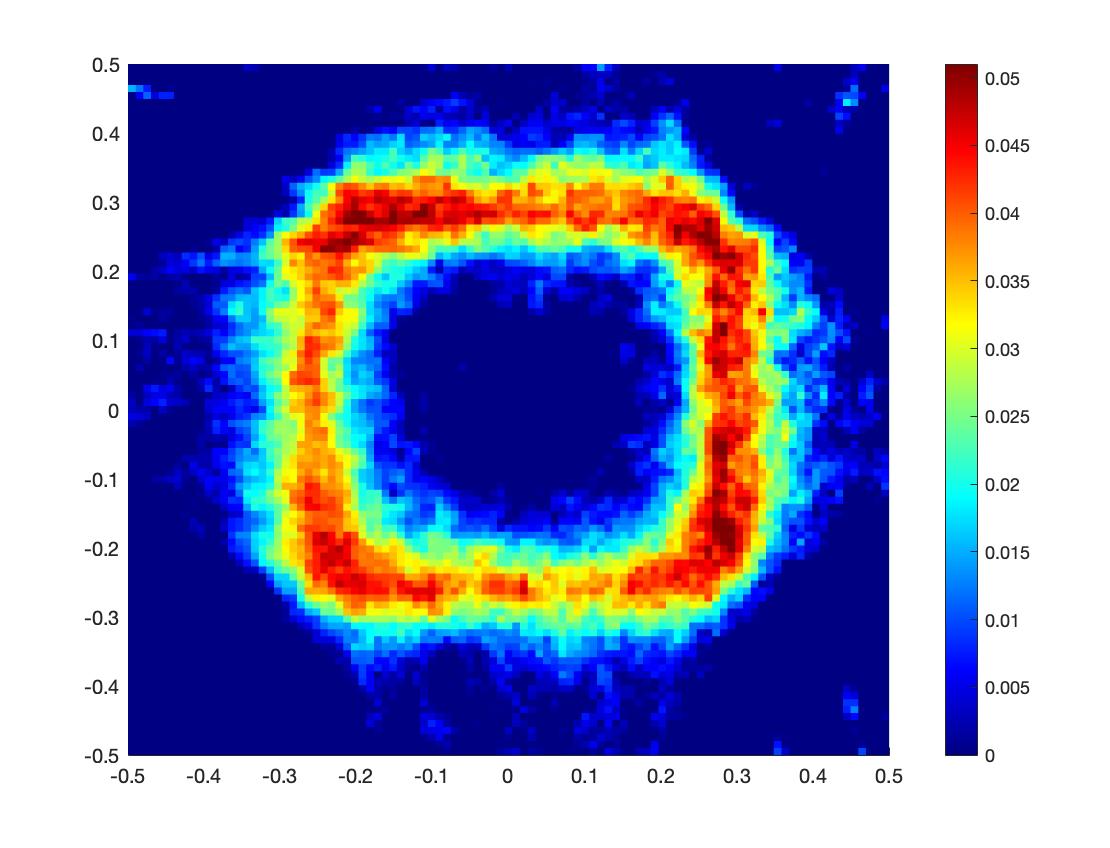}}

\caption{Example \ref{exp:5}.  Reconstructions of the slowness function with rectangular ring-shaped inhomogeneity  using exact and noisy measurement with noise level $1\%$ }
\end{figure}

\subsection{ Inhomogeneous velocity field of high contrast with  assumed background}
Now we examine the filtered back projection method \ref{Covsolve} with assumed background on two velocity models. For the reconstruction,  the measurement corresponding to 18 sources are collected at 153 points on the boundary with  $1\%$ noise.
\begin{example}\label{exp:6}
Consider the velocity model with a large rectangular nonhomogeneous regions in the assumed background. The assumed $\overline{f}$ satisfies  $\overline{f} =1.1$  inside the large obstacle of size $0.65\times 0.45$ centered at $(-0.125,-0.025)$ and $\overline{f}=1$ in the background as shown in \ref{d5}.  We are interested in resolving the two small rectangular shaped obstacles of size $0.1\times 0.1$ located at $(0.20,0)$ and $(-0.25,-0.25)$ as shown in Fig. \ref{d4}.  The slowness function inside these two small obstacles is taken to be   $1.05$ and $1.15$ respectively.
\end{example}

As one of the inclusions is within the large obstacle in the assumed background while the other one is outside of the obstacle, it is relatively hard to detect their locations precisely without the information of the assumed background due to the ill-posed nature of the inverse Eikonal tomography. From Fig. \ref{d6}, we can see that both small inclusions are well separated, and their locations are recovered pretty satisfactorily with the presence of $10\%$ noise in the measurement.

\begin{figure}[ht!]
\centering
\subfigure[Slowness function $f$]{
\label{d4}
\includegraphics[width=0.32\linewidth]{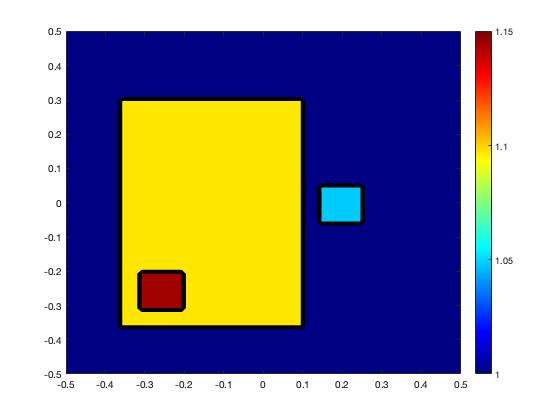}}
\subfigure[Assumed background $\overline{f}$]{
\label{d5}
\includegraphics[width=0.32\linewidth]{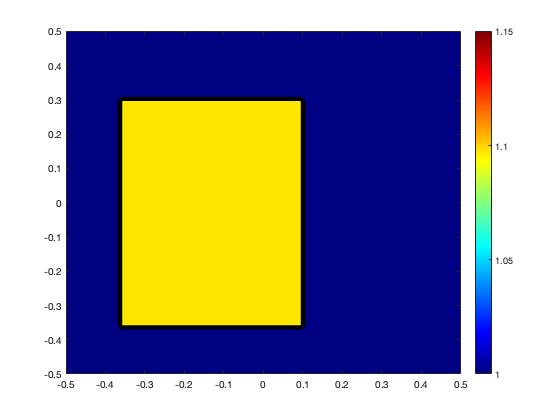}}
\subfigure[Reconstruction]{
\label{d6}
\includegraphics[width=0.32\linewidth]{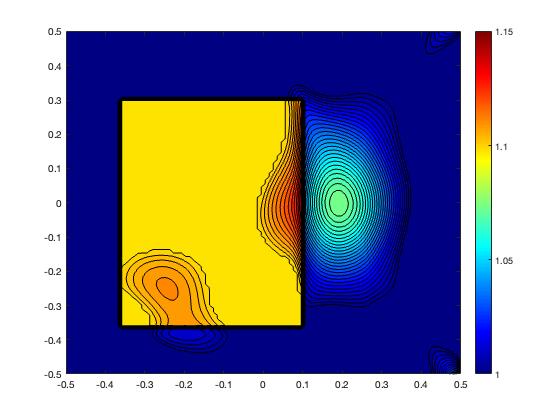}}
\caption{Example \ref{exp:6}. Reconstruction with an assumed discontinuous background slowness function $\overline{f}$}
\end{figure}

\begin{example}\label{exp:7}
In this example we consider the velocity model with the assumed background to be a continuous function \ref{j5}. The exact model in \ref{j4} is different from the assumed background in the regions marked with a white rectangle. 
\end{example}

The exact model of our interest has high contrast with variations of different scales, thus it is difficult to recover the regions related to small scales within the  regions marked with white rectangle.  We can observe from the reconstruction in \ref{j6} that the two inclusions are well separated and the locations captured agrees well with the exact velocity model upon noting the $10\%$ data noise.
\begin{figure}[ht!]
\centering
\subfigure[Slowness function $f$]{
\label{j4}
\includegraphics[width=0.32\linewidth]{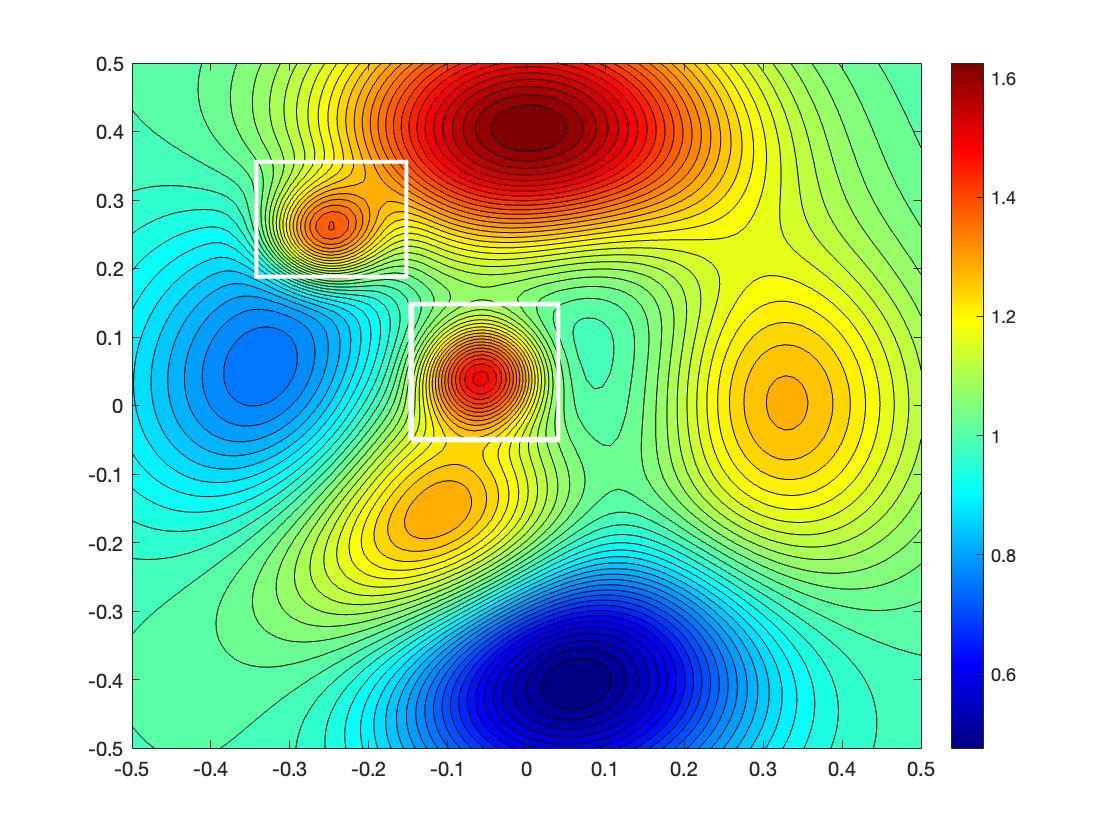}}
\subfigure[Assumed background $\overline{f}$]{
\label{j5}
\includegraphics[width=0.32\linewidth]{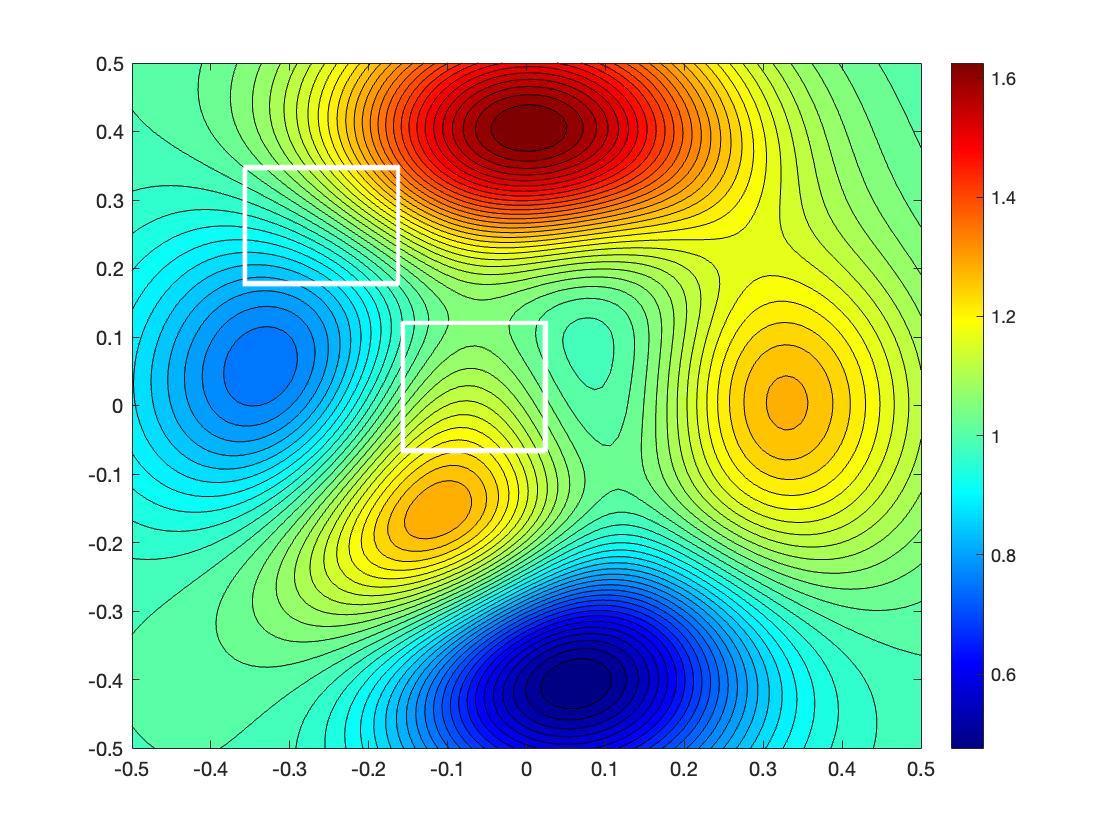}}
\subfigure[Reconstruction]{
\label{j6}
\includegraphics[width=0.32\linewidth]{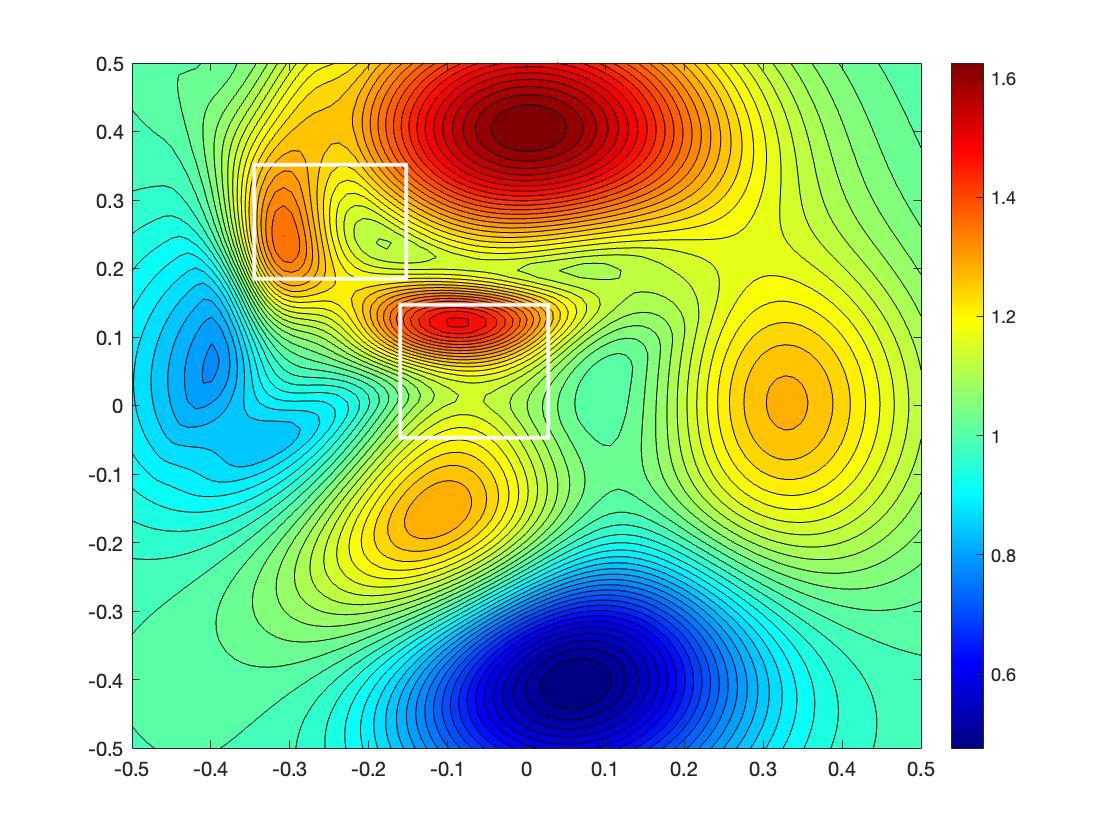}}
\caption{Example \ref{exp:7}. Reconstruction with an assumed discontinuous background slowness function
$\overline{f}$}
\end{figure}


\begin{thebibliography}{10}

\bibitem{erven1977RayMI}
V.~Cerveny, I.~A. Molotkov, and I.~Psencik.
\newblock Ray method in seismology.
\newblock 1977.

\bibitem{cerveny1985application}
V.~Cerveny.
\newblock The application of ray tracing to the numerical modeling of seismic
  wavefields in complex structures.
\newblock {\em Seismic shear waves}, 15:1--124, 1985.

\bibitem{cerveny1987ray}
V.~Cerveny.
\newblock Ray methods for three-dimensional seismic modelling.
\newblock {\em Lecture notes, Norwegian Institute of Technology, University of
  Trondheim}, 1987.

\bibitem{sei1994gradient}
Alain Sei and William~W Symes.
\newblock Gradient calculation of the traveltime cost function without ray
  tracing.
\newblock In {\em SEG Technical Program Expanded Abstracts 1994}, pages
  1351--1354. Society of Exploration Geophysicists, 1994.

\bibitem{sei1995convergent}
Alain Sei and William~W Symes.
\newblock Convergent finite-difference traveltime gradient for tomography.
\newblock In {\em SEG Technical Program Expanded Abstracts 1995}, pages
  1258--1261. Society of Exploration Geophysicists, 1995.

\bibitem{leung2006adjoint}
Shingyu Leung and Jianliang Qian.
\newblock An adjoint state method for three-dimensional transmission traveltime
  tomography using first-arrivals.
\newblock {\em Communications in Mathematical Sciences}, 4(1):249--266, 2006.

\bibitem{kak2001principles}
Avinash~C Kak and Malcolm Slaney.
\newblock {\em Principles of computerized tomographic imaging}.
\newblock SIAM, 2001.

\bibitem{natterer2001mathematics}
Frank Natterer.
\newblock {\em The mathematics of computerized tomography}.
\newblock SIAM, 2001.

\bibitem{helgason1980radon}
Sigurdur Helgason and S~Helgason.
\newblock {\em The radon transform}, volume~2.
\newblock Springer, 1980.

\bibitem{berryman1990stable}
James~G Berryman.
\newblock Stable iterative reconstruction algorithm for nonlinear traveltime
  tomography.
\newblock {\em Inverse problems}, 6(1):21, 1990.

\bibitem{bishop1985tomographic}
TN~Bishop, KP~Bube, RT~Cutler, RT~Langan, PL~Love, JR~Resnick, RT~Shuey,
  DA~Spindler, and HW~Wyld.
\newblock Tomographic determination of velocity and depth in laterally varying
  media.
\newblock {\em Geophysics}, 50(6):903--923, 1985.

\bibitem{washbourne2002crosswell}
John~K Washbourne, James~W Rector, and Kenneth~P Bube.
\newblock Crosswell traveltime tomography in three dimensions.
\newblock {\em Geophysics}, 67(3):853--871, 2002.

\bibitem{nolet2008breviary}
Guust Nolet.
\newblock A breviary of seismic tomography.
\newblock {\em A Breviary of Seismic Tomography}, 2008.

\bibitem{li2013fast}
Wenbin Li and Shingyu Leung.
\newblock A fast local level set adjoint state method for first arrival
  transmission traveltime tomography with discontinuous slowness.
\newblock {\em Geophysical Journal International}, 195(1):582--596, 2013.

\bibitem{li2014level}
Wenbin Li, Shingyu Leung, and Jianliang Qian.
\newblock A level-set adjoint-state method for crosswell
  transmission-reflection traveltime tomography.
\newblock {\em Geophysical Journal International}, 199(1):348--367, 2014.

\bibitem{taillandier2009first}
C{\'e}dric Taillandier, Mark Noble, Herv{\'e} Chauris, and Henri Calandra.
\newblock First-arrival traveltime tomography based on the adjoint-state
  method.
\newblock {\em Geophysics}, 74(6):WCB1--WCB10, 2009.

\bibitem{ito2013two}
Kazufumi Ito, Bangti Jin, and Jun Zou.
\newblock A two-stage method for inverse medium scattering.
\newblock {\em Journal of Computational Physics}, 237:211--223, 2013.

\bibitem{ito2022least}
Kazufumi Ito, Ying Liang, and Jun Zou.
\newblock Least-squares method for recovering multiple medium parameters.
\newblock {\em Inverse Problems}, 2022.

\bibitem{kruvzkov1967generalized}
SN~Kru{\v{z}}kov.
\newblock Generalized solutions of nonlinear first order equations with several
  independent variables. ii.
\newblock {\em Mathematics of the USSR-Sbornik}, 1(1):93, 1967.

\bibitem{lions1982generalized}
Pierre-Louis Lions.
\newblock {\em Generalized solutions of Hamilton-Jacobi equations}, volume~69.
\newblock London Pitman, 1982.

\bibitem{bardi1997optimal}
Martino Bardi, Italo~Capuzzo Dolcetta, et~al.
\newblock {\em Optimal control and viscosity solutions of
  Hamilton-Jacobi-Bellman equations}, volume~12.
\newblock Springer, 1997.

\bibitem{kwasnicki2017ten}
Mateusz Kwa{\'s}nicki.
\newblock Ten equivalent definitions of the fractional laplace operator.
\newblock {\em Fractional Calculus and Applied Analysis}, 20(1):7--51, 2017.

\bibitem{zeng2010medical}
Gengsheng~Lawrence Zeng.
\newblock {\em Medical image reconstruction: a conceptual tutorial.}
\newblock Springer, 2010.

\end{thebibliography}

\end{document}